\documentclass[11pt]{article}
\usepackage{graphicx}
\usepackage{mathrsfs}
\usepackage{multicol,graphics}
\usepackage{amsmath,amssymb}
\usepackage{footmisc}
\usepackage{amssymb}
\usepackage{mathtools}
\usepackage{amsthm,amsfonts}
\usepackage{epsf}
\usepackage{epstopdf}
\usepackage{pdfpages}
\usepackage{subfigure}
\usepackage{graphicx}
\usepackage[capposition=bottom]{floatrow}
\usepackage{color}
\usepackage{cite}

\usepackage{sidecap}
\usepackage{tikz}
\usepackage{caption}

\newtheorem{thm}{Theorem}[section]
\newtheorem{definition}{Definition}[section]
\newtheorem{lemma}[thm]{Lemma}
\newtheorem{remark}{Remark}[section]

\newtheorem{theorem}[thm]{Theorem}

\numberwithin{equation}{section}

\allowdisplaybreaks

\title{\textbf{  A time-periodic  competition model with nonlocal dispersal and bistable nonlinearity: propagation dynamics  and stability  }}

\author{Manjun Ma\footnote{Department of Mathematics, School of Science, Zhejiang Sci-Tech University, Hangzhou, Zhejiang, 310018, China. Email:\texttt{mjunm9@zstu.edu.cn}.}, \  \ \
 	   Wentao Meng\footnote{Department of Mathematics, School of Science, Zhejiang Sci-Tech University, Hangzhou, Zhejiang, 310018, China. Email: \texttt{202020102047@mails.zstu.edu.cn}.  }, \  \  \
  Chunhua Ou \footnote{Corresponding author. Department of Mathematics and Statistics, Memorial University of Newfoundland, St. John's, Newfoundland, Canada, A1C 5S7. Email: \texttt{ou@mun.ca.}},  \,\,\,
and \,\,\,
	   Jiajun Yue\footnote{Department of Mathematics, School of Science, Zhejiang Sci-Tech University, Hangzhou, Zhejiang, 310018, China. Email: \texttt{yjjminer@sina.com}.  }
}

\date{}

\begin{document}
	\maketitle
\begin{quote}
\textbf{Abstract}: Seasonality frequently occurs in population models,  and the corresponding seasonal patterns have been of   great interest to scientists.  This paper is  concerned with  traveling waves to a time-periodic bistable Lotka-Volterra competition system with nonlocal dispersal. We first establish the existence, uniqueness  and stability of traveling wave solutions for this system. Then, by utilizing  comparison principle and the stability property, the relationship among  the bistable wave speed,   the asymptotic propagation speeds of the associated monotone subsystems  and the speed of upper/lower solutions is obtained.  Next, explicit sufficient conditions for positive and negative bistable wave speeds are derived.  Our explicit results are derived by constructing  particular upper/lower solutions with specific  asymptotical behaviors,  which can be seen as case studies  shedding light on further studies and improvements. Finally, the theoretical results are corroborated under weak conditions by direct simulations of the underlying time-periodic system with nonlocal dispersal. The combined impact of competition, dispersal and seasonality on the invasion direction has shed new light on  the modelings and analysis of population competition and species invasion in heterogeneous media.
%Finally, by applying these theorems,  we obtain  explicit conditions that determine  the propagation direction of of the invasion.

\vspace{0.2in}

\end{quote}
	%\noindent \textbf{2010 Mathematics Subject Classification}: Primary 35K57, 35B40, 35C07.
	%
	%\vspace{0.2in} \noindent \textbf{Keywords and Phrases}: Traveling wave solutions, speed selection, upper and lower solutions method
	
\section{Introduction}\label{sec1}	
In this   paper we are concerned with traveling wave solutions to the following nonlocal dispersal system
\begin{equation}\label{original_model}
\left\{
\begin{array}{lr}
u_{t}=d_{1}(t)\,[J_{1}\ast u-u]+u(r_{1}(t)-a_{1}(t)u-b_{1}(t)v),\\
v_{t}=d_{2}(t)\,[J_{2}\ast v-v]+v(r_{2}(t)-a_{2}(t)u-b_{2}(t)v),
\end{array}
\text{$ x \in \mathbb{R},\, \,t>0,$}
\right.
\end{equation}
where $u=u(x,t)$ and $v=v(x,t)$ stand for the densities of two competitive species  at position $x$ at time $t;$ the functions $d_i(t)$ are the dispersal coefficients, $r_i(t)$ are the net birth rates or resource strength, $r_i(t)/a_i(t)$ are called the carrying capacities, $b_i(t)/r_i(t)$ are the competition coefficients, $J_i$ represent the kernel functions, \,\,$i=1,2.$ The convolution $J_{i}\ast \omega(x,t)$ means
$$
J_{i}\ast \omega(x,t)= \int_{\mathbb{R}}J_{i}(x-y)\omega(y,t) \,dy
$$
for any continuous function $\omega(x,t)$. Moveover, we assume that  all the coefficients are continuous positive T-periodic functions and satisfy the bistable nonlinearity
\begin{equation}
 \int_0 ^T a_1(t)p(t)-b_{1}(t)q(t) \ dt <0, \,\,\, \int_0 ^T b_2(t)q(t)-a_{2}(t)p(t) \ dt <0, \label{A2}
\end{equation}
where
\begin{equation}\label{p(t)}
\left\{
\begin{array}{lr}
p(t)=\dfrac{p_0 e^{\int_0 ^t r_1(s)ds}}{1+p_0 \int_0 ^t e^{\int_0 ^s r_1(\tau)d\tau}a_1(s)ds },\, \,\,\, p_0=\dfrac{e^{\int_0 ^T r_1(s)ds}-1}{\int_0 ^T e^{\int_0 ^s r_1(\tau)d\tau}a_1(s)ds }>0,\\[2mm]
q(t)=\dfrac{q_0 e^{\int_0 ^t r_2(s)ds}}{1+q_0 \int_0 ^t e^{\int_0 ^s r_2(\tau)d\tau}b_2(s)ds },\, \, \,\, q_0=\dfrac{e^{\int_0 ^T r_2(s)ds}-1}{\int_0 ^T e^{\int_0 ^s r_2(\tau)d\tau}b_2(s)ds }>0.
\end{array}
\right.
\end{equation}

Competition among species is an eternal topic in nature. In recent years, Lotka-Volterra type systems with nonlocal dispersal have been frequently applied  to describe the dynamic interactions between two competing  species, see e.g.,  \cite{Bao2,GB,GB2,FD,YJ,JF,YZ,XS,XH}. Among them,
 Yu and Yuan in \cite{YZ} established the existence of traveling wave solutions to a nonlocal dispersal competitive-cooperative system by using the Schauder's fixed-point theorem and a cross-iteration technique.  Li and Lin \cite{XS} proved the existence of traveling wavefronts in a nonlocal dispersal cooperative system with delays by the method of upper and lower solutions. Bao, Li and Shen \cite{Bao2} investigated  the existence and non-existence of space-periodic traveling wave solutions of a competition system with nonlocal dispersal and space-periodic coefficients by means of the monotone semiflow theory. When all the coefficients are constant, the system in (\ref{original_model})  becomes
\begin{equation}\label{constant_model}
\left\{
\begin{array}{lr}
u_{t}=d_{1}\,[J_{1}\ast u-u]+u(r_{1}-a_{1}u-b_{1}v),\\
v_{t}=d_{2}\,[J_{2}\ast v-v]+v(r_{2}-a_{2}u-b_{2}v),
\end{array}
\right.
x \in \mathbb{R},\,t>0,
\end{equation}
which has been studied in \cite{GB,JF,SP}. Zhang, Ma and Li  in \cite{GB} showed that the bistable traveling waves with nonzero speed are strictly monotone.
Pan and Lin \cite{SP} proved the existence of the traveling waves  of system (\ref{constant_model}) with speed $c>c^*$ for a critical speed $c^*$, by constructing upper and lower solutions and using a limiting  process. Fang and Zhao \cite{JF} showed that the minimal wave speed must be   the spreading speed,  when monostable-nonlinearity is assumed. Dynamics for related diffusive Lotka-Velterra competitive models have been extensively studied  in \cite{AO_2018_2,Conley,Gardner,Girardin,GY,Hosono_1989,Hosono_1998,Kan-on_1995,Kan-on_1997,Ma_2019}.

In this paper, we further study  traveling wave solutions of Lotka-Volterra competition model (\ref{original_model}) when seasonality is coupled with nonlocal dispersal.   Throughout this paper, we will use the notation
$$\overline{f}=\frac{1}{T}\int_0 ^T f(t) \,dt$$
to denote the average value of a function on the interval $[0,T]$ and always assume the following:\\
(\textbf{A1}) $J_{i}$ is nonnegative and Lebesgue measurable for each $i;$ \\
(\textbf{A2}) For any $\lambda \in \mathbb{R},$ $\int_{\mathbb{R}}J_{i}(x)e^{-\lambda x}\,dx<\infty;$\\
(\textbf{A3}) $\int_{\mathbb{R}}J_{i}(x)dx=1,~x\in \mathbb{R}.$

Under the condition (\ref{A2}), the corresponding kinetic system of (\ref{original_model})
\begin{equation}\label{kinetic_model}
\left\{
\begin{array}{lr}
u_{t}=u(r_{1}(t)-a_{1}(t)u-b_{1}(t)v),\\
v_{t}=v(r_{2}(t)-a_{2}(t)u-b_{2}(t)v)
\end{array}
\right.
\end{equation}
has three nonnegative $T$-period  solutions  $(0,0)$, $(p(t),0)$, $(0,q(t))$ and at least  one coexistence solution $(u^*(t),v^*(t))$, where $p(t)$ and $q(t)$ are explicitly given by (\ref{p(t)}) { and satisfy  $0<u^*(t)<p(t)$, $0<v^*(t)<q(t)$ for all $t\in\mathbb{R}^+$; it  further follows  that  the two semitrivial periodic solutions $(p(t),0)$ and $(0,q(t))$ are stable, and  $(0,0)$ is unstable.  The uniqueness and the linear unstability of the coexistence solution $(u^*(t),v^*(t))$ is assured by a strong  condition}
 \begin{equation}\label{1.12a}
\overline{r}_1<\min_{0\leq t\leq T}\left(\frac{b_1(t)}{b_2(t)}\right)\,\overline{r}_2, \,\,\, \overline{r}_2<\min_{0\leq t\leq T} \left(\frac{a_2(t)}{a_1(t)}\right) \, \overline{r}_1,
\end{equation}
see e.g. \cite{Bao}.   A detailed argument  of the above results can be found in \cite{Bao}.

To study the time-periodic traveling wave of (\ref{original_model}) connecting $(0,q(t))$ to $(p(t),0),$  we set
\[
\phi(x,t)=\dfrac{u(x,t)}{p(t)}\,\qquad\text{and}\,\qquad\psi(x,t)=\dfrac{q(t)-v(x,t)}{q(t)},
\]
which leads to a cooperative system of the form
\begin{equation}\label{coo sys}
\left\{
\begin{array}{lr}
\phi_{t}=d_{1}(t)\left(\int_\mathbb{R} J_{1}(y)\phi(x-y,t) dy -\phi\right)+\phi\left[\,a_{1}(t)\,p(t)\,(1-\phi)-b_{1}(t)\,q(t)\,(1-\psi)\,\right], \\[2mm]
\psi_{t}=d_{2}(t)\left(\int_\mathbb{R} J_{2}(y)\psi(x-y,t) dy -\psi\right)+(1-\psi)\left[\, a_{2}(t)\,p(t)\,\phi-b_{2}(t)\,q(t)\,\psi\,\right],\\[2mm]
(\phi(x,0),\psi(x,0))=(\phi_{0},\psi_{0}),
\end{array}
\right.
\end{equation}
where
\[
\phi_{0}=\dfrac{u(x,0)}{p(0)}\, \,\text{and}\,\,\psi_0=\dfrac{q(0)-v(x,0)}{q(0)} \,\, \text{{are nonnegative real functions}.}
\]
Under this setting,  the trivial solution $(0, 0)$ of the system (\ref{original_model}) becomes
$\alpha_{1}=(0,1),$  while the other three  solutions $(p(t),0)$,  $(0,q(t))$ and $(u^*(t),v^*(t))$ becomes $\beta=(1,1),$ $\textbf{o}=(0,0)$ and a positive solution  $(\hat{\phi}(t),\hat{\psi}(t)),$ respectively.  Therefore, studying the traveling wave connecting $(0,q(t))$ to $(p(t),0)$ is equivalent to the study of the traveling wave of (\ref{coo sys}) from $\textbf{o}=(0,0)$ to $\beta=(1,1).$
Here a traveling wave solution of (\ref{coo sys}) is a translation invariant solution of the form
\begin{equation}\label{def}
\phi(x,t)=\Phi(z,t),\qquad \psi(x,t)=\Psi(z,t),\qquad  z=x+ct,
\end{equation}
where $c$ is the bistable wave speed. Thus, $(\Phi(z,t),\Psi(z,t))$ must satisfy the following wave profile system
\begin{equation}\label{wave sys}
\begin{cases}
\Phi_{t}=d_{1}(t)\left(\int_\mathbb{R} J_{1}(y)\Phi(z-y,t) dy -\Phi\right)-{c}{\Phi}_{z}+{\Phi}\left[a_{1}(t)p(t)(1-{\Phi})-b_{1}(t)q(t)(1-{\Psi})\right],\\
\Psi_{t}=d_{2}(t)\left(\int_\mathbb{R }J_{2}(y)\Psi(z-y,t) dy -\Psi\right)-{c}{\Psi}_z+(1-{\Psi})\left[a_{2}(t)p(t){\Phi}-b_{2}(t)q(t){\Psi}\right],\\
(\Phi(z,t),\Psi(z,t))=(\Phi(z,t+T),\Psi(z,t+T)),\\
\end{cases}
\end{equation}
with the asympototic conditions
\begin{equation}
(\Phi,\Psi)(-\infty,t)=(0,0),\qquad\qquad(\Phi,\Psi)(\infty,t)=(1,1). \label{1.8}
\end{equation}
For convenience, we set $\omega(x,t;\omega_{0})=(\phi(x,t;\omega_{0}),\psi(x,t;\omega_{0}))$ with $\omega_{0}=(\phi_{0},\psi_{0})$ and  $\Gamma(x+ct,t)=(\Phi(x+ct,t),\Psi(x+ct,t)).$ Moveover, let
\begin{equation}
f_{1}(\phi,\psi,t)=\phi\left[\,a_{1}(t)\,p(t)\,(1-\phi)-b_{1}(t)\,q(t)\,(1-\psi)\,\right]
 \end{equation}
and
\begin{equation}
f_{2}(\phi,\psi,t)=(1-\psi)\left[\, a_{2}(t)\,p(t)\,\phi-b_{2}(t)\,q(t)\,\psi\,\right].
\end{equation}
Thus we can rewrite (\ref{coo sys}) into \begin{equation}\label{1.12}
\left\{
\begin{array}{lr}
\phi_{t}=d_{1}(t)\left(\int_\mathbb{R} J_{1}(y)\phi(x-y,t) dy -\phi\right)+f_{1}(\phi,\psi,t),\\[2mm]
\psi_{t}=d_{2}(t)\left(\int_\mathbb{R} J_{2}(y)\psi(x-y,t) dy -\psi\right)+f_{2}(\phi,\psi,t),\\[2mm]
(\phi(x,0),\psi(x,0))=(\phi_{0},\psi_{0})(x).
\end{array}
\quad\text{$ x \in \mathbb{R},\, \,t>0,$}
\right.
\end{equation}

Nonlocal dispersal is totally different from the classical local diffusion and the solution map cannot smooth the initial data. It also yields challengings in the solution compactness as well as in the studying of eigenvalue problems. In this paper, we first establish the existence, uniqueness, monotonicity and stability of the traveling waves. Since the sign of wave speed determines which species will win the competition (or  which species will die out),  more interestingly and importantly,  we will study how to obtain  criteria to determine the speed sign of the wave. Our results provide possible deep understandings on the combined impact of competition, dispersal and seasonality on the invasion direction,  which  sheds new light on  the modelings and analysis of population competition and species invasion in heterogeneous media.

%Due to the presence of non-local dispersal term, we construct the upper and lower solutions of the dynamical system (\ref{coo sys}) (see Lemmas \ref{lecc} and \ref{upper and lower lemma})  and the upper solution of the wave profile (\ref{wave sys}) (see Theorem \ref{TH2}) by developing different ideas from the known references.
{\begin{remark}
It is easy to check that condition (\ref{A2}) is weaker than (\ref{1.12a}). In what follows, we will prove  the existence of a bistable traveling wave solution under (\ref{1.12a}). However, condition (\ref{A2}) is enough for us to prove  the uniqueness of  the bistable traveling wave solution. We conjecture  that  condition (\ref{A2}) can  guarantee the existence of  a bistable traveling wave solution.
\end{remark}}

 The paper is organized as follows. In section 2, we prove the existence,  monotonicity, uniqueness, and stability of the time-periodic traveling wave. In section 3, we establish  the value range of the bistable wave speed, which indicates  the relationship among the bistable wave speed, the spreading speeds of monostable subsystems, and the speed of upper/lower solutions. In section 4, by constructing upper/lower solutions, we derive  explicit conditions to get the positive and negative wave speeds. Examples and numerical simulations are presented in Section 5. Section 6 is with conclusion and discussion.

\section{The bistable traveling wave}
\subsection{Preliminaries}\label{subsec1}
\textbf{Notation. } We suppose that $\mathbb{\chi}$ is an ordered Banach space with a norm $\|\cdot\|_{\mathbb{\chi}}$ and its  positive cone $\mathbb{\chi}^+$ is well defined. Assume that ${\mathrm{Int}}(\mathbf{\chi}^+)$ is not empty.  For any $\xi, \varsigma \in \mathbb{\chi}$, we say  $\xi\geqslant \varsigma$ if $\xi-\varsigma\in \chi^+$, $\xi>\varsigma$ if $\xi\geqslant\varsigma$ but $\xi\neq \varsigma$, and  $\xi\gg\varsigma$  if $\xi-\varsigma\in {\mathrm{Int}}(\mathbf{\chi}^+)$. A subset of  $\mathbb{\chi}$ is called totally unordered provided that no two elements are ordered.  Let $$
\mathcal{C}=\lbrace u \in C(\mathbb{R},\chi)\vert  u  \text{ is a  nondecreasing function} \rbrace
$$ and equip $\mathcal{C}$ with a compact open topology.
For any  $\rho, \varphi\in\mathcal{C}$, we define
$$
\rho\geq\varphi \quad \text{if} \quad \rho(x)\ge \varphi(x) \quad\text{for all}\quad  x\in \mathbb{R}.
$$
Similarly, we define $\rho>\varphi$ if  $\rho\geq\varphi$ but $\rho\neq\varphi$, and $\rho\gg\varphi$ if $\rho(x)\gg \varphi(x)$
for all  $x\in \mathbb{R}$. Let $\chi_\sigma=\{\gamma\in\chi: \mathbf{o}\leq \gamma \leq \sigma\}$ and $\mathcal{C}_\sigma=\{\varphi\in\mathcal{C}: \mathbf{o}\leq\varphi\leq\sigma\}$ for any $\sigma\in\chi$ with $\sigma>\mathbf{o}$, where $\mathbf{o}$ is the zero element in $\chi$ or $\mathcal{C}$.

Assume that $\beta\in  {\mathrm{Int}}(\mathbf{\chi}^+)$ and $Q$ maps $\mathcal{C}_\beta$ to $\mathcal{C}_\beta$. Let $E$ be the set of all fixed pints of $Q$ restricted on $\chi_\beta$. Suppose that $\mathbf{o}$ and $\beta$ are in $E$.
Define a translation operator $T_{y}$ on $\mathcal{C}$ for any $y \in R$ by $ T_{y}[\phi](x)=\phi(x-y),\forall x \in \mathbb{R},\phi \in \mathcal{C}.$ Based on the idea in  \cite{Bao,Fang,JF},  we give the assumptions on the map $Q$.
%\begin{flushleft}
%\textbf{Assumptions on Q:}
%\end{flushleft}
\begin{enumerate}
\item[(H1)](Translation invariance) $T_{y}\circ Q[\phi]=Q\circ T_{y}[\phi],~\forall \phi \in \mathcal{C}_\beta,~y \in \mathbb{R}.$
\item[(H2)](Continuity) $Q: \mathcal{C}_\beta\rightarrow \mathcal{C}_\beta $ is continuous in the sense that if $\phi_{n}\rightarrow\phi$ in $\mathcal{C}_\beta$, then $	Q[\phi_{n}](x)\rightarrow Q[\phi](x)$ in $\chi_{\beta}$ for almost all $x \in \mathbb{R}.$
\item[(H3)](Monotonicity) $Q$ is order-preserving in the sense that $Q[\phi]\geqslant Q[\psi]$ whenever $\phi\geqslant\psi$ in $\mathcal{C}_\beta.$
\item[(H4)](Weak-compactness) For any fixed $x \in \mathbb{R},$ the set $Q\,[\,\mathcal{C_{\beta}}]$ is precompact in $\chi_{\beta}$.
\item[(H5)] (Bistability) Two fixed points $\mathbf{o}$ and $\beta$ are strongly stable from above and below, respectively. For the map $Q:\chi_\beta\rightarrow \chi_{\beta},$  the equilibra set $E \backslash \lbrace\mathbf{o},\beta\rbrace$ is totally unordered. The definition of strong stability is seen in \cite{Bao,Fang}.
\item[(H6)] (Counter-propagation) For $\alpha_{i} \in E \backslash \lbrace\mathbf{o},\beta\rbrace, c^*_-(\alpha_i,\beta)+c^*_+(0,\alpha_i)>0,i=1,2,$ where $ c^*_-(\alpha_i,\beta)$ and $c^*_+(0,\alpha_i)$ are recalled in \cite{Fang,Liang06,Liang07,Bao}.
\end{enumerate}
\begin{definition}\label{traveling wave def}
A family  of mappings  $\{ Q_t \}_{t\in \mathbb{R}^+}$ is called a $T$-periodic semiflow on space $\mathcal{C}$ provided that it has the following properties:
 \begin{enumerate}
\item[(i)] $Q_0[\phi]=\phi$, $\forall \phi\in\mathcal{C}$,
\item[(ii)] $Q_{t+T}[\phi]=Q_t\circ Q_T[\phi]$ for all $t\geq 0$, $\phi\in\mathcal{C}$,
\item[(iii)] $Q_{t_{n}}[\phi_{n}](x)\rightarrow Q_{t}\,[\phi](x)$  in $\chi_{\beta}$ for almost all $x \in \mathbb{R}$ whenever $t_{n}\rightarrow t$ and $\phi_{n}\rightarrow \phi$ in $\mathcal{C_{\beta}}.$
\end{enumerate}
\end{definition}
Moreover, the mapping $Q_T$ is called the Poincar\'{e} map associated with this periodic semiflow.
\begin{definition}(see \cite{Fang} Definition 3.3)
$\Gamma(x+ct,t)$ is said to be a traveling wave of the  semiflow $\{ Q_t \}_{t \in \mathbb{R^{+}}}$ with speed $c$, if $ Q_t [\Gamma(x,0)](x)=\Gamma(x + ct,t)$ for all $x \in \mathbb{R}$ and $ t \in \mathbb{R^{+}}.$
\end{definition}

\begin{lemma}\label{existence lemma}
(see \cite{Fang} Theorem 5.4) Let $\beta(t)$ be a strongly positive periodic fixed point of $\left\lbrace Q_{t}\right\rbrace _{t\geqslant0}$ restricted on $\chi_{\beta}$ with $Q_t[\beta(0)]=\beta(t)$ and  assume that $\left\lbrace Q_{t}\right\rbrace _{t\geqslant0}$ is a T-periodic semiflow on $\mathcal{C}_{\beta(0)}.$ Further, assume that the Poincar\'{e} map $Q_{T}$ satisfies  (H1)-(H6) with $\beta=\beta(0).$ Then there exist $c \in \mathbb{R}$ and $\phi(x,t)$ with $\phi(-\infty, t)=0$ and $\phi(+\infty ,t)=\beta(t)$ such that $Q_{t}[\phi](x)=\phi(x+ct, t)$ for all $(x,t) \in \mathbb{R}\times \mathbb{R^+} .$ Furthermore, $\phi(x,t)\in \mathcal{C}_{\beta(t)}$ is nondecreasing in $x$ and is $T$-periodic in $t.$
\end{lemma}

To discuss the dynamical behaviors of the semiflow generated by  system (\ref{coo sys}),  we first introduce the definition of upper and lower solutions to the wave profile system in (\ref{wave sys}).  An upper solution/lower solution of (\ref{coo sys}) can be defined similarly.
\begin{definition}\label{upper_solution_def1}
A pair of bounded function $(\Phi(z,t),\Psi(z,t))$ on $\mathbb{R}\times[0,T)$ is called an upper solution ( a lower solution) of (\ref{wave sys}) if $(\Phi(z,t),$ $\Psi(z,t))$ is continuous in $(z,t )\in \mathbb{R} \times [0,T)$, and satisfy
\[
\left\{
\begin{array}{lr}
\Phi_{t}\geq (\le)~ d_{1}(t)~[ J_{1}\ast\Phi(z,t)-\Phi(z,t)] -{c}{\Phi}_{z}+{\Phi}[a_{1}(t)\,p(t)(1-{\Phi})-b_{1}(t)\,q(t)(1-{\Psi})],\\
\Psi_{t}\geq (\le)~d_{2}(t)~[ J_{1}\ast\Psi(z,t)-\Psi(z,t)]-{c}{\Psi}_z+(1-{\Psi})(a_{2}(t)\,p(t){\Phi}-b_{2}(t)\,q(t){\bar\Psi})
\end{array}
\right.
\]
for all $(z,t)\in\mathbb{R}\times(0,T)$.
\end{definition}
\subsection{Existence and  monotonicity}
Let $\chi=\mathbb{R}^2$.  Let $P_{i}(t)$ be the solution semigroup of the linear nonlocal dispersal equation $u_{t}=d_{i}(t)[(J\ast u)-u]$ as below:
$$
P_{i}(t)[\phi](x)=e^{-\int_{0} ^{t}d_{i}(\tau)\,d\tau}\left[a_{0}+a_{1}\int_{0} ^{t}d_{i}(\tau)\,d\tau +a_{2}\int_{0} ^{t}\int_{0} ^{\tau}d_{i}(\tau)d_{i}(s)\,d\tau \,ds+\cdot\cdot\cdot \right](x),\,\,i=1,2,
$$
where $a_{0}(\phi)=\phi$ and $a_{m}(\phi)=J_{i}\ast a_{m-1}(\phi)$\,, $\forall \,m\geqslant1.$ At this point, define
\[
P(t)=\left( \begin{matrix}
                 P_{1}(t) & 0 \\
                 0 & P_{2}(t)
            \end{matrix}\right),   \qquad
f(\omega,t)= \left( \begin{matrix}
                 f_{1}(\phi,\psi,t)\\
                 f_{2}(\phi,\psi,t)
            \end{matrix}\right).
\]
Then the solution of system (\ref{coo sys})(or \ref{1.12})  can be represented in the integral form
\begin{equation} \label{integral solution}
\omega(x,t;\omega_{0})=P(t)[\omega_{0}](x)+\int_0^t P(t-s)[f(\omega(\cdot,s),s)](x)ds,\,x \in \mathbb{R}, \,t\geqslant0.
\end{equation}
By this, we define a family of operators $Q_{t}$ associated with  system (\ref{coo sys}) by
\begin{equation}\label{flow}
Q_{t}(\omega_{0})=\omega(x,t;\omega_{0}),\quad \forall\, x \in \mathbb{R},\,t\geqslant0.
\end{equation}
{It is easy to show that $ Q_{t}(\omega_{0})$ is a $T$-periodic semiflow.} However,
the existence of a bistable traveling wave is usually difficult to prove. Here we use the theory of  monotone dynamical systems developed in \cite{Fang} to deal with it. Hence  a further condition on  the symmetry of the kernel functions  is   required so that the counter-propagation (H6) is satisfied for  the Poincar\'{e} map  $Q_{T}$ associated with (\ref{flow}), i.e.,
$$
Q_{T}(\omega_{0})=\omega(x,T;\omega_{0})=P(T)[\omega_{0}](x)+\int_0^T P(T-s)[f(\omega(\cdot,s),s)](x)ds,\,x\in \mathbb{R},\, \omega_{0}\in \mathcal{C}_\beta.
$$

\begin{theorem}\label{existence}
Assume that $J_i(x)=J_i(-x), i=1,2$ and (\ref{1.12a}) holds. Then there exist a  constant $c \in \mathbb{R}$ and a  T-periodic nondecreasing (in $z$) traveling wave profile $\Gamma(z,t)=(\Phi(z,t),\Psi(z,t))$  to (\ref{wave sys})-(\ref{1.8}), where $ z=x+ct, \ \Gamma(z,t+T)=\Gamma(z,t)$.  Moreover,  $\dfrac{\partial}{\partial z}\Phi_{\pm}(z,t)>0$ and $\dfrac{\partial}{\partial z}\Psi_{\pm}(z,t)>0$ for $z \in \mathbb{R}$ and $t \in \mathbb{R_{+}}.$ Here, for $\pm$, we mean the left and right derivatives at $z$.
\end{theorem}
\begin{proof}
We can easily verify that $Q_{T}$ satisfies assumptions (H1)-(H5). If (H6) is true for  $Q_{T}$, then Lemma \ref{existence lemma} guarantees the first statement in the theorem.  In the following, we prove that $Q_{T}$ satisfies (H6), that is,
\begin{equation}\label{h6}
c^*_-(\alpha_i,\beta)+c^*_+(0,\alpha_i)>0, i=1, 2,
\end{equation}
where  $c^{*}_{-}(\alpha_{i},\beta)$ is called the leftward spreading speed of $Q_{T}$ in the phase space  $\mathcal{C}_{[\alpha_{i}, \beta]}$, and   $c^{*}_{+}(0,\alpha_{i})$ is called the rightward spreading speed of $Q_{T}$ in the phase space  $\mathcal{C}_{[0,\alpha_{i}]}$ (see (2.3) in \cite{Bao} or (2.8) in \cite{Fang}).
%was proved in for the case of $i=2$.

Here we only  prove
inequality (\ref{h6}) for the case of $i=1$, i.e., $\alpha_{1}=(0,1)$, since the other case ($i=2$) can be similarly handled as in \cite{Bao} with the assumption of (\ref{1.12a}). Suppose that $(\Phi(x+ct,t), \Psi(x+ct,t))$ is a traveling wave solution of (\ref{coo sys}) connecting $\alpha_1$ to $\beta$. Then $(\phi(x,t), \psi(x,t))$ solves
\begin{equation}\label{c01}
\left\{
\begin{array}{lr}
\phi_{t}=d_{1}(t)\left(\int_\mathbb{R} J_{1}(y)\phi(x-y,t) dy -\phi(x,t)\right)+\,a_{1}(t)\,p(t)\phi\,(1-\phi),\\
\psi_{t}=0.
\end{array}
\right.
\end{equation}
Linearizing the first equation at $\phi=0$, we have
\begin{equation}\label{fir-linear}
\phi_{t}=d_{1}(t)\left(\int_\mathbb{R} J_{1}(y)\phi(x-y,t) dy -\phi(x,t)\right)+\,a_{1}(t)\,p(t)\phi.
\end{equation}
Let the solution of (\ref{fir-linear}) be of the form $\eta_{1}(t)e^{\mu x}$. Then $\eta_{1}(t)$ satisfies the $\mu$-parameterized linear equation
\begin{equation}\label{c01a}
\left\{
\begin{array}{lr}
\eta_{1}'(t)= \left( d_{1}(t)\int_\mathbb{R} J_{1}(y)e^{-\mu y} dy - d_{1}(t)+a_{1}(t)p(t)\right) \eta_{1}(t),\\
\eta_{1}(0)=\eta_{1}(T).
\end{array}
\right.
\end{equation}
It is well known that the principal eigenvalue of (\ref{c01}) is
$$
\gamma_{1}(\mu)=\dfrac{1}{T}\int_0 ^T d_{1}(t)\left(\int_\mathbb{R} J_{1}(y)e^{-\mu y} dy-1\right)+a_{1}(t)p(t)dt.
$$
Furthermore, by the reference \cite{Liang07}, we have
\begin{equation}\label{C-}
c^{*}_{-}(\alpha_{1},\beta)=\inf_{0<\mu<\infty}\frac{\gamma_{1}(\mu)}{\mu}.
\end{equation}
The condition $J_1(x)=J_1(-x)$ implies that it is positive. On the other hand, assume that  $(\Phi(x+ct,t), \Psi(x+ct,t))$ is a traveling wave solution of (\ref{coo sys}) connecting $0$ to $\alpha_1$. Then, by  (\ref{coo sys}), it is obvious that $(\phi(x,t), \psi(x,t))$ satisfies
\begin{equation}\label{coo sys1}
\left\{
\begin{array}{lr}
\phi_{t}=0,\\
\psi_{t}=d_{2}(t)\left(\int_\mathbb{R }J_{2}(y)\psi(x-y,t) dy -\psi(x,t)\right)-b_{2}(t)\,q(t)\,\psi(1-\psi).
\end{array}
\right.
\end{equation}
Repeating the above process, we obtain
\begin{equation}\label{C+}
c^{*}_{+}(0,\alpha_{1})=\inf_{0<\mu<\infty}\frac{\gamma_{2}(\mu)}{\mu}>0,
\end{equation}
where
$$
\gamma_{2}(\mu)=\frac{1}{T}\int_0 ^T d_{2}(t)\left(\int_\mathbb{R} J_{2}(y)e^{\mu y} dy-1\right)+b_{2}(t)q(t) dt.
$$
By (\ref{C-}) and (\ref{C+}), we get that (H6) is true.

Next, we verify the second statement in the theorem. By Lemma \ref{existence lemma}, it follows that
$$
\dfrac{\partial}{\partial z}\Phi(z,t)\geq0 \  \text{ and }  \ \  \dfrac{\partial}{\partial z}\Psi(z,t)\geq0
$$
for $z \in \mathbb{R}$ and $t \in \mathbb{R^{+}} $ if the derivatives  exist. If they do not exist, here we mean the left and right derivatives.  We need only to prove that the equal sign does not appear.
Suppose that there exists $z_{0}, t_0 \in \mathbb{R}$ such that $\dfrac{\partial}{\partial z}\Phi(z_0,t_0)=0$. By the periodic property of the wave functions, it would give   $\dfrac{\partial}{\partial z}\Phi(z_0,t_0+nT)=0$ for any positive integer $n$. We can assume that $t_0=0$ and $z_0=x_0+ct_0=x_0$ for some $x_0$. Therefore, we can re-write (\ref{1.12}), with initial wavefront profile,  as
 \begin{equation}\label{1.12abc}
\left\{
\begin{array}{lr}
\phi_{t}=d_{1}(t)\left(\int_\mathbb{R} J_{1}(y)\phi(x-y,t) dy -\phi\right)-\beta_1 \phi+\bar f_{1}(\phi,\psi,t),\\[2mm]
\psi_{t}=d_{2}(t)\left(\int_\mathbb{R} J_{2}(y)\psi(x-y,t) dy -\psi\right)-\beta_2 \psi+\bar f_{2}(\phi,\psi,t),\\[2mm]
(\phi(x,0),\psi(x,0))=(\Phi(x,0),\Psi(x,0))
\end{array}
\quad\text{$ x \in \mathbb{R},\, \,t>0,$}
\right.
\end{equation}
Here $\bar f_1=f_1+\beta_1 \phi, \bar f_2=f_2+\beta_2 \psi$ with a proper choice of $\beta_1$ and $\beta_2$ so that both $\bar f_1$ and $\bar f_2$ are monotone in $\phi$ and $\psi$. Taking derivative with respect to $x$ at both sides of each equation in (\ref{1.12abc}) gives

\begin{equation}\label{1.12aa}
\left\{
\begin{array}{lr}
(\phi_x)_{t}=d_{1}(t)\left(\int_\mathbb{R} J_{1}(y)\phi_x(x-y,t) dy -\phi_x\right)-\beta_1 \phi_x+\bar f_{1\phi}\phi_x+\bar f_{1\psi}\psi_x ,\\[2mm]
(\psi_x)_{t}=d_{2}(t)\left(\int_\mathbb{R} J_{2}(y)\psi_x(x-y,t) dy -\psi_x\right)-\beta_2 \psi_x+\bar f_{2\phi}\phi_x+\bar f_{2\psi}\psi_x,\\[2mm]
(\phi_x(x,0),\psi_x(x,0))=(\Phi_{x}(x,0),\Psi_{x}(x,0)),
\end{array}
\quad\text{$ x \in \mathbb{R},\, \,t>0,$}
\right.
\end{equation}
where $g_{iy}$ represents the partial derivative of  $g_{i}$ with respective to $y$. Let $\bar P_{i}(t)$ be the solution semigroup of the linear nonlocal dispersal equation $u_{t}=d_{i}(t)[(J\ast u)-u]-\beta_i u$. As in (\ref{integral solution}), we get from (\ref{1.12aa})
 \begin{equation} \label{integral solutiona}
\phi_x(x,t)=\bar P_1(t)[\Phi_x(x,0)](x)+\int_0^t \bar P_1(t-s)[\bar f_{1\phi}\phi_x+\bar f_{1\psi}\psi_x]ds\ge \bar P_1(t)[\Phi_x(x,0)](x).
\end{equation}

  Recall that the support of $J_1(x)$ contains at least an interval with the length large than  zero,  and then this makes $\bar P_1(t)[\omega_x(x,0)](x)$ positive for  $x=x_0$ when time $t$ is large, say $nT$ for large $n$. This is a contradiction.
   Thus the supposition is false, and the proof is complete.
\end{proof}

\begin{remark}
When the wave speed $c$ is not zero, it can be proved that $\dfrac{\partial}{\partial z}\Phi(z,t)$ and $\dfrac{\partial}{\partial z}\Psi(z,t)$ are continuous functions in $(z,t)$. However, the smooth property is not clear to us when $c=0$.
\end{remark}

\subsection{Uniqueness}
The following comparison principle can be proved by properly modifying the argument of Lemma 3.2 in \cite{HR}. Hence the proof is omitted here.
\begin{lemma}\label{comparison}
Suppose that $(\phi^{-},\psi^{-})(x,t)$ and $(\phi^{+},\psi^{+})(x,t)$ in $\mathcal{C_{\beta}}$ are a bounded lower solution and a bounded upper solution of (\ref{coo sys}) on $\mathbb{R}\times[0,T)$. Then we have
\item[\rm(i):] if $\phi^{-}(x,0)\leqslant \phi^{+}(x,0)$ and $\psi^{-}(x,0)\leqslant \psi^{+}(x,0)$ for $x \in \mathbb{R},$ then
\begin{equation}
\phi^{-}(x,t)\leqslant \phi^{+}(x,t),~~~\psi^{-}(x,t)\leqslant \psi^{+}(x,t), ~~~~~(x,t) \in \mathbb{R}\times[0,\infty).
\end{equation}
 \item[\rm(ii):]  if $\phi^{-}(x,0)\leqslant \phi_{0}\leqslant \phi^{+}(x,0)$ and $\psi^{-}(x,0)\leqslant \psi_{0}\leqslant \psi^{+}(x,0)$ for $x \in \mathbb{R}$, where $(\phi_{0},\psi_{0})=w_{0}$ is the initial data of (\ref{coo sys}) then
\begin{equation}
\begin{cases}
\phi^{-}(x,t)\leqslant \phi(x,t;w_{0})\leqslant \phi^{+}(x,t),\\
\psi^{-}(x,t)\leqslant \psi(x,t;w_{0})\leqslant \psi^{+}(x,t),\\
\end{cases}
~~~~(x,t) \in \mathbb{R}\times[0,\infty).
\end{equation}
\end{lemma}
 To proceed, we need the following lemma.
 \begin{lemma}\label{lecc}
 Assume that (\ref{A2}) holds. There exist two  positive pairs  $(\lambda_{0}, (p_{1}^{-}(t),p_{2}^{-}(t)))$ and $(\lambda_{1}, (p_{1}^{+}(t),p_{2}^{+}(t)))$ solving the following eigenvalue inequality problems
\begin{equation}\label{inequ1}
\left\{
\begin{array}{lr}
\dfrac{d p_{1}^{-}(t)}{d t}\ge \left[\,a_{1}(t)\,p(t)\,-b_{1}(t)\,q(t)\,+\lambda_{0}\right]p_{1}^{-}(t), \\[2mm]
\dfrac{d p_{2}^{-}(t)}{d t}\ge a_{2}(t)\,p(t)\,p_{1}^{-}(t)\,+\left[\lambda_{0}-b_{2}(t)\,q(t)\right]\,p_{2}^{-}(t),\\[2mm]
 p_{1}^{-}(t+T)= p_{1}^{-}(t),~ p_{2}^{-}(t+T)= p_{2}^{-}(t)
\end{array}
\right.
\end{equation}
and
\begin{equation}\label{inequ2}
\left\{
\begin{array}{lr}
\dfrac{d p_{1}^{+}(t)}{d t}\ge \left[\lambda_{1}-a_{1}(t)\,p(t)\right]\,p_{1}^{+}(t)+b_{1}(t)\,q(t)\,p_{2}^{+}(t), \\[2mm]
\dfrac{d p_{2}^{+}(t)}{d t}\ge \left[\lambda_{1}+b_{2}(t)\,q(t)\,-a_{2}(t)\,p(t)\,\right]p_{2}^{+}(t), \\[2mm]
 p_{1}^{+}(t+T)= p_{1}^{+}(t),~ p_{2}^{+}(t+T)= p_{2}^{+}(t),
\end{array}
\right.
\end{equation}
respectively.
\end{lemma}
\begin{proof}
We only prove (\ref{inequ1}) since the proof of (\ref{inequ2}) is similar  and omitted here.
Take
$$
0<\lambda_{0}<\min \{\overline{ b_{1}(t)q(t)-a_{1}(t)p(t)}, \overline{b_2(t)q(t)} \},
$$
and
\begin{equation*}
\left\{
\begin{array}{lr}
p_{1}^{-}(t)=\text{exp}\left(\int_{0}^t a_{1}(\tau)\,p(\tau)\,-b_{1}(\tau)\,q(\tau)\,d\tau+\overline{b_1q-a_1p} \,t\right),\\
p_{2}^{-}(t)=\left(c_{0}(t)+p_{2}^{-}(0)\right)\text{exp}\left(-\int_0^t b_{2}(\tau)q(\tau)d\tau +\lambda_{0}t\right),\\[2mm]
\end{array}
\right.
\end{equation*}
where
\begin{equation*}
\left\{
\begin{array}{lr}
p_{1}^{-}(0)=1,\\
p_{2}^{-}(0)=\dfrac{\int_0^T a_{2}(t)\,p(t)\,p_{1}^{-}(t)\text{exp}\left(\int_0^t b_{2}(\tau)q(\tau)d\tau -\lambda_{0}t\right) dt}{\text{exp}\left(\int_0^T b_{2}(t)q(t)dt -\lambda_{0}T\right)-1}, \\[4mm]
c_{0}(t)=\int_{0}^{t} a_{2}(s)p(s)p_{1}^{-}(s)\text{exp}\left(\int_0^s b_{2}(\tau)q(\tau)d\tau-\lambda_{0}s\right) \,ds.
\end{array}
\right.
\end{equation*} Then it is easy to check that  (\ref{inequ1}) is true.
\end{proof}
We next apply the eigenvalues $\lambda_0, \lambda_1 $ and eigenfunctions $(p_{1}^{-}(t),p_{2}^{-}(t))$ and $(p_{1}^{+}(t),p_{2}^{+}(t))$ to construct upper and lower solutions of the system  (\ref{coo sys}).
\begin{lemma}\label{upper and lower lemma}
 Assume that (\ref{A2}) holds and  there exists  $(c, \Phi(z,t),\Psi(z,t))$ as a traveling wave solution of (\ref{coo sys}).
  Then there exist positive constants $\sigma_{1}$, $ \rho$, $\delta$, real numbers $\kappa^{\pm} \in \mathbb{R}$  and positive  functions $p_{1}(x,t), p_{2}(x,t)$
such that
\begin{eqnarray}\label{upper and lower}
&\phi^{\pm}(x,t)=\Phi\left(x+ct+\kappa^{\pm}\pm\sigma_{1}\delta(1-e^{-\rho t}),t\right)\pm\delta p_{1}\left(x+ct+\kappa^{\pm}\pm\sigma_{1}\delta(1-e^{-\rho t}),t\right)e^{-\rho t},\nonumber\\
&\psi^{\pm}(x,t)=\Psi\left(x+ct+\kappa^{\pm}\pm\sigma_{1}\delta(1-e^{-\rho t}),t\right)\pm\delta p_{2}\left(x+ct+\kappa^{\pm}\pm\sigma_{1}\delta(1-e^{-\rho t}),t\right)e^{-\rho t}\nonumber
\end{eqnarray}
  are upper-lower solutions of (\ref{coo sys}) for  $(x,t)\in\mathbb{R}\times(0,\infty)$.
\end{lemma}
\begin{proof}
Define a continuous function  $\zeta(x)$ by
$$
\zeta(x)=\begin{cases}
~0,~~~x<-M,\\
~1,~~~x>M,
\end{cases}
$$
where $M$ is a large  positive constant and $0\leqslant \zeta'(x)\leqslant 1$ for $x \in \mathbb{R}.$ Furthermore, define $p(x,t)=(p_{1}(x,t),p_{2}(x,t))$ as follows:
$$
p_{1}(x,t)=\zeta(x)p_{1}^{+}(t)+(1-\zeta(x))p_{1}^{-}(t),$$
$$
p_{2}(x,t)=\zeta(x)p_{2}^{+}(t)+(1-\zeta(x))p_{2}^{-}(t)$$
and
$$\xi^{\pm}(x,t)=x+ct+\kappa^{\pm}\pm\sigma_{1}\delta(1-e^{-\rho t}).$$
Then we can re-write
\begin{equation}\label{special upper and lower}
\begin{array}{lr}
\phi^{\pm}(x,t)=\Phi(\xi^{\pm},t)\pm\delta p_{1}(\xi^{\pm},t)e^{-\rho t},\\
\psi^{\pm}(x,t)=\Psi(\xi^{\pm},t)\pm\delta p_{2}(\xi^{\pm},t)e^{-\rho t}.
\end{array}
\end{equation}
We only show that $(\phi^{+}(x,t),\psi^{+}(x,t))$ is an upper solution of system (\ref{coo sys}). The proof of the lower solution is similar and is omitted here.
To this end, several  notations are first given by
\[d=\max_{t \in [0,T]}\left\lbrace d_{1}(t),d_{2}(t)\right\rbrace,\]
\[C_{0}=\max\left\lbrace  \max_{\xi^+ \in \mathbb{R},t \in [0,T]}\vert\dfrac{\partial}{\partial t}p_{1}(\xi^{+},t)\vert,\quad \max_{\xi^+ \in \mathbb{R},t \in [0,T]}\vert\dfrac{\partial}{\partial t}p_{2}(\xi^{+},t)\vert\right\rbrace,  \]
\[C_{1}=\min\left\lbrace  \inf_{\xi^+ \in [-M,M],t \in [0,T]}\dfrac{\partial \Phi}{\partial \xi^+}(\xi^+,t),\inf_{\xi^+ \in [-M,M],t \in [0,T]}\dfrac{\partial \Psi}{\partial \xi^+}(\xi^+,t)\right\rbrace,\]
\[C_{2}=\max\left\lbrace  \max_{t \in [0,T]}(a_{1}(t)p(t)+b_{1}(t)q(t)),\max_{t \in [0,T]}(a_{2}(t) p(t)+b_{2}(t)q(t))\right\rbrace,\]
\[C_{3}=\max\left\lbrace  \max_{t \in [0,T]}\vert p_{1}^{+}(t)-p_{1}^{-}(t)\vert, \max_{t \in [0,T]}\vert p_{1}^{+}(t)+p_{1}^{-}(t)\vert, \max_{t \in [0,T]}  \vert p_{2}^{+}(t)-p_{2}^{-}(t)\vert, \max_{t \in [0,T]}\vert p_{2}^{+}(t)+p_{2}^{-}(t)\vert\right\rbrace.\]
We shall prove the result on three value intervals of $\xi^{+}(x,t) \in \mathbb{R}$.
\item[(i):] $\xi^{+}(x,t)\leqslant -M.$ According to the definition of the function $\zeta(x),$ we have
$$\zeta(\xi^{+})=0, \ p_{1}(\xi^{+},t)=p_{1}^{-}(t) \ \text{ and} \   \ p_{2}(\xi^{+},t)=p_{2}^{-}(t).$$
Then
$$
\phi^{+}(x,t)=\Phi(\xi^{+},t)+\delta p_{1}^{-}(t)e^{-\rho t},\quad
\psi^{+}(x,t)=\Psi(\xi^{+},t)+\delta p_{2}^{-}(t)e^{-\rho t}.
$$
Substituting these expressions into the first equation of system (\ref{coo sys}) and using Lemma \ref{lecc} give
\begin{eqnarray*}
\begin{aligned}
&d_{1}(t)[J_{1}*\phi^{+}(x,t)-\phi^{+}(x,t)]-\phi^{+}_{t}+\phi^{+}[\,a_{1}(t)\,p(t)\,(1-\phi^{+})-b_{1}(t)\,q(t)\,(1-\psi^{+})] \\
=d_{1}(t)&\left\lbrace \int_\mathbb{R} J_{1}(y)[\Phi(\xi^{+}-y,t)+\delta p_{1}^{-}(t)e^{-\rho t}] dy -\Phi(\xi^{+},t)-\delta p_{1}^{-}(t)e^{-\rho t}\right\rbrace \\
&-\left[\Phi_{\xi^{+}}\xi^{+}_{t}+\Phi_{t}+\delta e^{-\rho t}\dfrac{d p_{1}^{-}(t)}{d t}-\delta \rho e^{-\rho t}p_{1}^{-}(t)\right] +\left(\Phi(\xi^{+},t)+\delta p_{1}^{-}(t)e^{-\rho t}\right)[\,a_{1}(t)\,p(t)\,\\
&\left(1-\Phi(\xi^{+},t)-\delta p_{1}^{-}(t)e^{-\rho t}\right)-b_{1}(t)\,q(t)\,\left(1-\Psi(\xi^{+},t)-\delta p_{2}^{-}(t)e^{-\rho t}\right)] \\
\leq d_{1}(t)&\int_\mathbb{R} J_{1}(y)[\Phi(\xi^{+}-y,t)-\Phi(\xi^{+},t)]dy-c\Phi_{\xi^{+}}-\Phi_{t}+\Phi[\,a_{1}(t)\,p(t)\,(1-\Phi)\\
&-b_{1}(t)\,q(t)\,(1-\Psi)]-\sigma_{1} \delta \rho e^{-\rho t}\Phi_{\xi^{+}}-\delta e^{-\rho t}[\,a_{1}(t)\,p(t)\,-b_{1}(t)\,q(t)\,+\lambda_{0}]p_{1}^{-}(t)\\
&+\delta \rho e^{-\rho t}p_{1}^{-}(t) +\,(\Phi+\delta p_{1}^{-}(t)e^{-\rho t})[\,a_{1}(t)\,p(t)\,(1-\Phi-\delta p_{1}^{-}(t)e^{-\rho t})\\
&-b_{1}(t)\,q(t)\,(1-\Psi-\delta p_{2}^{-}(t)e^{-\rho t})]-\Phi[\,a_{1}(t)\,p(t)\,(1-\Phi)-b_{1}(t)\,q(t)\,(1-\Psi)]\\
=-\delta e^{-\rho t}&\left(\sigma_{1} \rho \Phi_{\xi^{+}}+[\,a_{1}(t)\,p(t)\,-b_{1}(t)\,q(t)\,+\lambda_{0}-\rho]p_{1}^{-}(t)+\,\Phi[a_{1}(t)\,p(t)p_{1}^{-}(t)-b_{1}(t)\,q(t)p_{2}^{-}(t)]\right)\\
&+\delta e^{-\rho t}\left(a_{1}(t)\,p(t)\,(1-\Phi-\delta e^{-\rho t}p_{1}^{-}(t))-b_{1}(t)\,q(t)\,(1-\Psi-\delta e^{-\rho t}p_{2}^{-}(t))\right)p_{1}^{-}(t)\\
\stackrel{def}{=}\Pi.&&
\end{aligned}
\end{eqnarray*}
Here the last equality holds by using
\begin{equation}\label{c1}
d_{1}(t)\int_\mathbb{R} J_{1}(y)[\Phi(\xi^{+}-y,t)-\Phi(\xi^{+},t)] dy -c\Phi_{\xi^{+}}-\Phi_{t}+\Phi[\,a_{1}(t)\,p(t)\,(1-\Phi)-b_{1}(t)\,q(t)\,(1-\Psi)]=0.
\end{equation}
Recall that when $M$ is sufficiently large (i.e., $\xi^{+}(x,t)$ is negative enough), $(\Phi(\xi^{+},t),\Psi(\xi^{+},t))\rightarrow (0,0),$  and let $\rho$ be small enough to have
$$
\Pi\rightarrow \delta p_{1}^{-}(t)e^{-\rho t}\left\lbrace \rho-\lambda_{0}+\delta e^{-\rho t}[b_{1}(t)\,q(t)p_{2}^{-}(t)-a_{1}(t)\,p(t)p_{1}^{-}(t)]\right\rbrace\leqslant0.
$$
For the second equation in system (\ref{coo sys}), by Lemma \ref{lecc}, we have
\begin{eqnarray*}
\begin{aligned}
&d_{2}(t)[J_{2}*\psi^{+}(x,t)-\psi^{+}(x,t)]-\psi^{+}_{t}+(1-\psi^{+})\left[\,a_{2}(t)\,p(t)\,\phi^{+}-b_{2}(t)\,q(t)\,\psi^{+}\,\right] \\
=&d_{2}(t)\left\lbrace \int_\mathbb{R} J_{2}(y)[\Psi(\xi^{+}-y,t)+\delta p_{2}^{-}(t)e^{-\rho t}] dy -\Psi(\xi^{+},t)-\delta p_{2}^{-}(t)e^{-\rho t}\right\rbrace \\
&-\left(\Psi_{\xi^{+}}\xi^{+}_{t}+\Psi_{t}+\delta e^{-\rho t}\dfrac{d p_{2}^{-}(t)}{d t}-\delta \rho e^{-\rho t}p_{2}^{-}(t)\right) +\,\left(1-\Psi(\xi^{+},t)-\delta p_{2}^{-}(t)e^{-\rho t}\right)\\
&\left(\,a_{2}(t)\,p(t)\,[\Phi(\xi^{+},t)+\delta p_{1}^{-}(t)e^{-\rho t}]-b_{2}(t)\,q(t)\,[\Psi(\xi^{+},t)+\delta p_{2}^{-}(t)e^{-\rho t}]\right),\\
\leq&d_{2}(t)\left\lbrace \int_\mathbb{R} J_{2}(y)[\Psi(\xi^{+}-y,t) -\Psi(\xi^{+},t)] dy\right\rbrace-{c}{\Psi}_{\xi^{+}}-\Psi_{t}+(1-{\Psi})[a_{2}(t)p(t){\Phi}\\
&-b_{2}(t)q(t){\Psi}]-\sigma_{1} \delta \rho e^{-\rho t}\Psi_{\xi^{+}}-\delta e^{-\rho t}[a_{2}(t)\,p(t)\,p_{1}^{-}(t)\,+(\lambda_{0}-b_{2}(t)\,q(t))\,p_{2}^{-}(t)]\\
&+\delta \rho e^{-\rho t}p_{2}^{-}(t)+(1-\Psi)[a_{2}(t)\,p(t)\,(\Phi(\xi^{+},t)+\delta p_{1}^{-}(t)e^{-\rho t})-b_{2}(t)\,q(t)\,(\Psi(\xi^{+},t)\\
&+\delta p_{2}^{-}(t)e^{-\rho t})]-\delta  e^{-\rho t}p_{2}^{-}(t)[a_{2}(t)\,p(t)\,(\Phi(\xi^{+},t)+\delta p_{1}^{-}(t)e^{-\rho t})-b_{2}(t)\,q(t)\,(\Psi(\xi^{+},t)\\
&+\delta p_{2}^{-}(t)e^{-\rho t})]-(1-{\Psi})[a_{2}(t)p(t){\Phi}-b_{2}(t)q(t){\Psi}]\\
\rightarrow &\delta p_{2}^{-}(t)e^{-\rho t}\left\lbrace \rho-\lambda_{0}+\delta e^{-\rho t}[b_{2}(t)\,q(t)p_{2}^{-}(t)-a_{2}(t)\,p(t)p_{1}^{-}(t)]\right\rbrace\leqslant0. \\
\end{aligned}
\end{eqnarray*}
The above inequality holds when $\rho$ is small enough.
\item[(ii):] $-M\leqslant\xi^{+}(x,t)\leqslant M.$ By observing the definition of $C_{3},$ we have
 $$
\mid \dfrac{\partial p_{1}(\xi^{+},t)}{\partial \xi^{+}}\mid=\mid \dfrac{\partial \zeta(\xi^{+},t)}{\partial \xi^{+}}\mid\cdot\mid p_{1}^{+}(t)-p_{1}^{-}(t)\mid\leqslant \mid p_{1}^{+}(t)-p_{1}^{-}(t)\mid\leqslant C_{3},  $$
   $$
\mid \dfrac{\partial p_{2}(\xi^{+},t)}{\partial \xi^{+}}\mid=\mid \dfrac{\partial \zeta(\xi^{+},t)}{\partial \xi^{+}}\mid\cdot\mid p_{2}^{+}(t)-p_{2}^{-}(t)\mid\leqslant \mid p_{2}^{+}(t)-p_{2}^{-}(t)\mid\leqslant C_{3},$$
   $$
\mid p_{1}(\xi^{+},t)\mid=\mid \zeta(\xi^{+}) p_{1}^{+}(t)+(1-\zeta(\xi^{+}))p_{1}^{-}(t)\mid\leqslant \mid p_{1}^{+}(t)+p_{1}^{-}(t)\mid\leqslant C_{3},$$
   $$
\mid p_{2}(\xi^{+},t)\mid=\mid \zeta(\xi^{+}) p_{2}^{+}(t)+(1-\zeta(\xi^{+}))p_{2}^{-}(t)\mid\leqslant \mid p_{2}^{+}(t)+p_{2}^{-}(t)\mid\leqslant C_{3}.$$
We now give the following estimates in terms of   $C_{2}$ and $C_{3}$.
\begin{eqnarray*}
\begin{aligned}
&\phi^{+}[\,a_{1}(t)\,p(t)\,(1-\phi^{+})-b_{1}(t)\,q(t)\,(1-\psi^{+})]-\Phi[\,a_{1}(t)\,p(t)\,(1-\Phi)-b_{1}(t)\,q(t)\,(1-\Psi)]\\
=&\delta e^{-\rho t}p_{1}(\xi^{+},t)\lbrace a_{1}(t)p(t)-b_{1}(t)q(t)+b_{1}(t)\,q(t)\Psi-2a_{1}(t)p(t)\Phi-a_{1}(t)p(t)\delta e^{-\rho t}p_{1}(\xi^{+},t)\rbrace\\
&+\delta e^{-\rho t}p_{2}(\xi^{+},t)b_{1}(t)q(t)\lbrace\Phi+\delta e^{-\rho t}p_{1}(\xi^{+},t)\rbrace\\
\leq& \delta e^{-\rho t}C_3\lbrace 3a_{1}(t)p(t)+3 b_{1}(t)q(t)+ a_{1}(t)p(t)\delta e^{-\rho t} C_3+b_{1}(t)q(t)\delta e^{-\rho t}C_3\rbrace\\
\leq& \delta e^{-\rho t}C_3C_2( 3+2\delta C_3).\\
\end{aligned}
\end{eqnarray*}
The second estimate is as follows:
\begin{eqnarray*}
\begin{aligned}
&(1-\psi^{+})\left[\,a_{2}(t)\,p(t)\,\phi^{+}-b_{2}(t)\,q(t)\,\psi^{+}\,\right]-(1-\Psi)\left[\,a_{2}(t)\,p(t)\,\Phi-b_{2}(t)\,q(t)\,\Psi\,\right] \\
=&\left[1-\Psi-\delta e^{-\rho t}p_{2}(\xi^{+},t)\right]a_{2}(t)\,p(t)\,\delta e^{-\rho t}p_{1}(\xi^{+},t)+\lbrace 2\,b_{2}(t)\,q(t)\left(\Psi+ \frac{1}{2}\delta e^{-\rho t}p_{2}(\xi^{+},t)\right)\\
&-b_{2}(t)\,q(t)-a_{2}(t)\,p(t)[\Phi+\delta e^{-\rho t}p_{1}(\xi^{+},t)]\rbrace \delta e^{-\rho t}p_{2}(\xi^{+},t)\\
\leq& \delta e^{-\rho t}C_3\lbrace2C_2+b_{2}(t)\,q(t)+a_{2}(t)\,p(t)\delta C_3+C_2C_3\delta\rbrace\\
\leq& \delta e^{-\rho t}C_3C_2( 3+2\delta C_3).
\end{aligned}
\end{eqnarray*}
By using  $p_{1}(\xi^{+},t)=\zeta(\xi^{+})p_{1}^{+}(t)+(1-\zeta(\xi^{+}))p_{1}^{-}(t)$, we have the third estimate:
\begin{eqnarray*}
\begin{aligned}
&\vert d_{1}(t)\left\lbrace \int_{\mathbb{R}}J_{1}(y)p_{1}(\xi^+-y,t)dy-p_{1}(\xi^+,t)\right\rbrace \vert \\
=&\vert d_{1}(t)\left\lbrace \int_{\mathbb{R}}J_{1}(y)[\zeta(\xi^{+}-y)p_{1}^{+}(t)+p_{1}^{-}(t)-\zeta(\xi^{+}-y)p_{1}^{-}(t)]dy-\zeta(\xi^{+})p_{1}^{+}(t)-p_{1}^{-}(t)
+\zeta(\xi^{+})p_{1}^{-}(t)\right\rbrace \vert \\
=&\vert d_{1}(t)\left\lbrace \int_{\mathbb{R}}J_{1}(y)[\zeta(\xi^{+}-y)p_{1}^{+}(t)-\zeta(\xi^{+}-y)p_{1}^{-}(t)]dy-\zeta(\xi^{+})p_{1}^{+}(t)+\zeta(\xi^{+})p_{1}^{-}(t)\right\rbrace \vert \\
=&\vert d_{1}(t)\vert \cdot \vert\left[ \int_{\mathbb{R}}J_{1}(y)\zeta(\xi^{+}-y)dy-\zeta(\xi^{+})\right]p_{1}^{+}(t)- \left[\int_{\mathbb{R}}J_{1}(y)\zeta(\xi^{+}-y)dy-\zeta(\xi^{+})\right]p_{1}^{-}(t) \vert \\
\leq & 2\,d\,(p_{1}^{+}(t)+p_{1}^{-}(t))\leq 2\,d\,C_{3}.
\end{aligned}
\end{eqnarray*}
Similarly, it follows that
$$
\vert d_{2}(t)\left\lbrace \int_{\mathbb{R}}J_{2}(y)p_{2}(\xi^{+}-y,t)-p_{2}(\xi^{+},t)dy\right\rbrace \vert \leqslant 2\,d\,C_{3}.
$$
Consequently, applying the above estimates and (\ref{c1}), and further taking $\delta\leq \frac{C_1}{2C_3}$ and
$$
\sigma_{1}\geqslant \dfrac{2C_{0}+2C_{3}(3C_{2}+\vert c \vert+\rho+2d+C_{2}C_{1})}{C_{1}\rho},
$$
we have
\begin{eqnarray*}
\begin{aligned}
&d_{1}(t)[J_{1}*\phi^{+}(x,t)-\phi^{+}(x,t)]-\phi^{+}_{t}+\phi^{+}\left[\,a_{1}(t)\,p(t)\,(1-\phi^{+})-b_{1}(t)\,q(t)\,(1-\psi^{+})\,\right] \\
=&d_{1}(t)\left[\int_{\mathbb{R}}J_{1}(y)(\Phi(\xi^{+}-y,t)+\delta p_{1}(\xi^{+}-y,t)e^{-\rho t})dy-(\Phi(\xi^{+},t)+\delta p_{1}(\xi^{+},t)e^{-\rho t})\right]\\
&-\left\lbrace\Phi_{\xi^{+}}\xi_{t}+\Phi_{t}+\delta e^{-\rho t}\left[ \dfrac{d p_{1}(\xi^{+},t)}{d \xi}(c+\sigma_{1}\rho \delta e^{-\rho t})+\dfrac{d p_{1}(\xi^{+},t)}{d t}\right]+\delta e^{-\rho t} (-\rho) p_{1}(\xi^{+},t)\right\rbrace\\
&+\phi^{+}[\,a_{1}(t)\,p(t)\,(1-\phi^{+})-b_{1}(t)\,q(t)\,(1-\psi^{+})\,] \\
=&\phi^{+}[\,a_{1}(t)\,p(t)\,(1-\phi^{+})-b_{1}(t)\,q(t)\,(1-\psi^{+})]-\Phi[\,a_{1}(t)\,p(t)\,(1-\Phi)-b_{1}(t)\,q(t)\,(1-\Psi)]\\
&+d_{1}(t)\left\lbrace \int_{\mathbb{R}}J_{1}(y)\left(p_{1}(\xi^{+}-y,t)-p_{1}(\xi^{+},t)\right)dy\right\rbrace \delta e^{-\rho t}-\Phi_{\xi^{+}}\sigma_{1}\rho \delta e^{-\rho t}\\ &-\delta e^{-\rho t}\left\lbrace \dfrac{d p_{1}(\xi^{+},t)}{d \xi}(c+\sigma_{1}\rho \delta e^{-\rho t})+\dfrac{d p_{1}(\xi^{+},t)}{d t}\right\rbrace +\delta e^{-\rho t} \rho p_{1}(\xi^{+},t)\\
\leqslant &\delta e^{-\rho t} \left( C_3C_2( 3+2\delta C_3)+2\,d\,C_{3}-C_{1}\sigma_{1}\rho+\vert c \vert \, C_{3}+\sigma_{1}\rho\delta C_{3} e^{-\rho t}+C_{0}+\rho C_{3}\right)\leq 0.
\end{aligned}
\end{eqnarray*}
Likewise, it can be verified that
\begin{eqnarray*}
\begin{aligned}
&d_{2}(t)[J_{2}*\psi^{+}(x,t)-\psi^{+}(x,t)]-\psi^{+}_{t}+(1-\psi^{+})\left[\,a_{2}(t)\,p(t)\,\phi^{+}-b_{2}(t)\,q(t)\,\psi^{+}\,\right]\\
\leqslant &\delta e^{-\rho t} \left( C_3C_2( 3+2\delta C_3)+2\,d\,C_{3}-C_{1}\sigma_{1}\rho+\vert c \vert \, C_{3}+\sigma_{1}\rho\delta C_{3} e^{-\rho t}+C_{0}+\rho C_{3}\right)\leq 0.
\end{aligned}
\end{eqnarray*}
 For case \rm(iii): $\xi^{+}(x,t)\geqslant M$, the result can be  proved by using the same method as in case \rm(i). Hence $(\phi^{+},\psi^{+})$ is an upper solution of system (\ref{coo sys}). Then the proof is complete.
 \end{proof}

 \begin{remark}
{ System (\ref{coo sys}) is monotone only in the phase space
$$\mathbb{W}=\{(\phi, \psi)|\phi, \psi\in \mathcal{C}, \phi\ge 0 \ \text{and} \ \psi\le 1 \}.$$ When  $\phi^-\not\in \mathbb{W}$ or $\psi^+\not\in \mathbb{W}$,  we  can use their truncations $\hat \phi^-=\max \{0, \phi^-\}$ and $\hat \psi^+=\min \{1, \psi^+ \}$ to replace $\phi^-$ and $\psi^+$, respectively, so that the comparison principle in Lemma \ref{comparison} still works for the upper and lower solutions in Lemma \ref{upper and lower lemma}.}
 \end{remark}
We are in a position now to state and prove the uniqueness of the bistable time-periodic traveling wave {if it exists.}
\begin{theorem}\label{uniquenes}
Assume that (\ref{A2}) holds. Then there exists at most one (up to translation) bistable  time-periodic traveling wave solution to (\ref{wave sys})-(\ref{1.8}).
\end{theorem}
\begin{proof}
We first prove the uniqueness of the bistable wave speed by contradiction and assume that (\ref{wave sys})-(\ref{1.8}) has
two solutions $(\Phi,\Psi)(x+ct,t)$  and $(\Phi_{1},\Psi_{1})(x+c_{1}t,t)$ with speeds $c$ and $ c_1$.
By Lemma \ref{upper and lower lemma} and the comparison principle, we have
\begin{equation}\label{27}
\begin{array}{lr}
\Phi(x+ct+\kappa^{-}-\sigma_{1}\delta(1-e^{-\rho t}),t)-Ce^{-\rho t}/2 =\phi^-(x,t)\leq\Phi_{1}(x+c_{1}t,t)
\end{array}
\end{equation}
for some  $\kappa^{-}\in \mathbb{R}$ and $C>0$ since (\ref{27}) is true at $t=0$.  It follows from the above formulas that $c\leqslant c_{1}$. Otherwise, assume that $\Phi_{1}(\eta,t)<1$ for some fixed value $\eta$ and $t\in (0,\infty)$. If $c\geqslant c_{1}$, then on the line $x+c_{1}t=\eta$ we have
\begin{equation*}
\Phi_{1}(\eta,t)=\Phi_{1}(x+c_{1}t,t)\geqslant \phi^-(x,t)\geqslant \Phi\left(x+c_{1}t+(c-c_1)t+\kappa^{-}-\sigma_{1}\delta(1-e^{-\rho t}),t\right)-Ce^{-\rho t}/2.
\end{equation*}
 It turns out that $\Phi_{1}(\eta,t)\geqslant 1$ as $t$ is sufficiently large.  This is a contradiction.  By using the same idea we can also prove $c_{}\geq c_{1}$. Therefore, we have $c=c_{1}$.

 Now from (\ref{27}) by letting $t\to \infty$, we get
 \[
 \Phi(\eta+\kappa^{-}-\sigma_1\delta,t )\leq\Phi_{1}(\eta,t).\]

 Similarly we can get

  \[
 \Phi_1(\eta,t)\leq\Phi(\eta+\kappa^{+}+\sigma_1\delta,t).\]
 As such, we can easily follow the idea in \cite{Chen1997} (see Step 2 on  page 133) to prove that the wave profile is unique up to translation.

\end{proof}

\subsection{Stability  }
This subsection is devoted to discussing the Liapunov stability of the bistable $T$-periodic traveling wave solution of system  (\ref{coo sys}).
\begin{theorem} \label{stability}
 Assume that (\ref{1.12a}) holds and there exists   $\Gamma(z,t)=(\Phi(z,t), \Psi(z,t)), z=x+ct$  as the bistable  $T$-periodic traveling wave profile of system (\ref{coo sys}) connecting $\textbf{o}$ and $\beta.$ Suppose  that $\omega(x,t)=(\phi(x,t),\psi(x,t))$ is the solution of system (\ref{coo sys}) with the initial data $\omega_{0}=(\phi_{0},\psi_{0})$ satisfying $(0,0)\le \omega_{0}\le (1,1)$.  Then the traveling wave $\Gamma(z,t)$ is stable in the sense that for arbitrary small $\epsilon>0$ there is a constant $\delta^*$,
such that
\begin{equation}\label{re}
\| \, \omega(x,t)-\Gamma(z,t) \, \| <\epsilon, \  (x,t)\in \mathbb{R}\times \mathbb{R}^+
\end{equation} as long as
$\omega_{0}$  satisfies
\begin{equation}\label{th1.2-cond}
\|\omega_{0}(x)-\Gamma(x,0)\|<\delta^*, \ x\in \mathbb{R}.
\end{equation}
\end{theorem}
\begin{proof} Let  $0<\delta\leq \frac{C_1}{2C_3}$  be defined in Lemma \ref{upper and lower lemma} with $\kappa^{\pm}=0$, $\delta_{0}=\inf_{x\in \mathbb{R}}\min \left\lbrace p_{1}(x,0),p_{2}(x,0)\right\rbrace $ and  $\delta^*=\delta \delta_0$.
Then condition (\ref{th1.2-cond}) means that
$$
\Phi(x,0)-\delta p_{1}(x,0)\leqslant \phi_{0} \leqslant \delta p_{1}(x,0)+\Phi(x,0), \ x\in \mathbb{R}
$$
and
$$
\Psi(x,0)-\delta p_{2}(x,0)\leqslant \psi_{0} \leqslant \delta p_{2}(x,0)+\Psi(x,0), \ x\in \mathbb{R}.
$$
By this and (\ref{special upper and lower}), we have
$$
\phi^{-}(x,0) \leqslant \phi_{0}(x)\leqslant \phi^{+}(x,0), \quad \psi^{-}(x,0) \leqslant \psi_0(x)\leqslant \psi^{+}(x,0).
$$
Then lemma \ref{comparison} implies that
\begin{equation}\label{re2}
\phi^{-}(x,t) \leqslant \phi(x,t) \leqslant \phi^{+}(x,t), \quad \psi^{-}(x,t) \leqslant \psi(x,t)\leqslant \psi^{+}(x,t),
\end{equation}
where $(\phi^{\pm}(x,t),\psi^{\pm}(x,t))$ are defined in lemma \ref{upper and lower lemma}. Thus we have
\begin{eqnarray*}
\begin{aligned}
&\vert \, \phi^{\pm}(x,t)-\Phi(z,t) \, \vert\\
 \leq &\vert \,  \delta p_1\left(z+ \sigma_1 \delta(1-e^{-\rho t}),t\right)e^{-\rho t} \, \vert+\vert \, \Phi\left(z+ \sigma_1 \delta(1-e^{-\rho t}),t\right)-\Phi(z,t) \, \vert\\
 \leq & \delta\vert \, p_1\left(z+ \sigma_1 \delta(1-e^{-\rho t}),t\right) \, \vert e^{-\rho t}+\sigma_1\delta\left\vert \, \frac{\partial }{\partial z}\Phi\left(z+ \theta\sigma_1 \delta(1-e^{-\rho t}),t\right) \, \right\vert (1-e^{-\rho t})\\
 \leq & \chi \delta,
\end{aligned}
\end{eqnarray*}
where $\theta\in (0,1)$ and $\chi>0$ does not depend on $\delta$. Likewise,  it is easy to check that
$$
\vert \, \psi^{\pm}(x,t)-\Psi(z,t) \, \vert\leq \chi \delta.
$$
By (\ref{re2}) and taking  $\delta< \min\{\frac{C_1}{2C_3}, \frac{\epsilon}{\sqrt{2}\chi}\}$, i.e., $\delta^*<\delta_0\min\{\frac{C_1}{2C_3}, \frac{\epsilon}{\sqrt{2}\chi}\}$, after a simple computation, we get
  $$
\| \, \omega(x,t)-\Gamma(z,t) \, \| \leq\sqrt{2} \chi\delta<\epsilon.
$$
The proof is complete.
\end{proof}

\section{The value interval of the bistable wave speed}
In this section, we focus on the study of the speed of the  bistable time-periodic traveling wave (if it exists) with a  value interval estimate and its sign determination. We will assume that only (\ref{A2}) (instead of (\ref{1.12a}) ) holds in the following sections.

\begin{theorem} \label{th3.1}   Let $c$ be the speed of the bistable $T$-periodic  traveling wave solution of (\ref{coo sys}), connecting $\mathbf{o}$ to $\beta $. Then we have
\begin{equation}\label{speed1}
		-c^{*}_{+}(0,\alpha_{i})\le c\le c^{*}_{-}(\alpha_{i},\beta),\;\;\;i=1,2.
\end{equation}
Particularly, for $i=1$ we have
\begin{equation}\label{speed2}
\inf_{0<\mu<\infty}\frac{\gamma_1(\mu)}{\mu}\le c\le \inf_{0<\mu<\infty}\frac{\gamma_2(\mu)}{\mu},
\end{equation}
where
$$
\gamma_1(\mu)= \dfrac{1}{T}\int_0 ^T d_{1}(t)\left(\int_R J_{1}(y,t)e^{-\mu y} dy-1\right)+a_{1}(t)p(t)dt
$$
and
$$
\gamma_2(\mu)= \frac{1}{T}\int_0 ^T d_{2}(t)\left(\int_R J_{2}(y,t)e^{\mu y} dy-1\right)+b_{2}(t)q(t) dt.
$$
\end{theorem}
\begin{proof}
We shall prove $c\le c^*_{-}(\alpha_1,\beta)$. The remainder of  (\ref{speed1}) can be proved by means of a similar method.

Let $(\Phi _1,\Psi _1)(z,t)$ be the $T$-periodic monostable traveling wave profile of (\ref{wave sys}) satisfying
\[
(\Phi _1,\Psi _1)(-\infty,t)=\alpha_1,\;\;(\Phi _1,\Psi _1)(\infty,t)=\beta
\;
\]
with the wave speed $c_{-}^{*}(\alpha _1,\beta)$. Then, $(\Phi_1,\Psi _1)(x+c_{-}^{*}(\alpha _1,\beta)t,t)$ is an exact solution of (\ref{coo sys}) with the initial data as $(\Phi _1,\Psi _1)(x,0)$. To proceed, we give another initial functions $(\phi,\psi)(x,0)$ of (\ref{coo sys}), which is continuous, nondecreasing and satisfies
\begin{equation}\label{initial_data1}
\phi(x,0)=\psi(x,0)=
\begin{cases}
~0,~~~x<-L,\\
1-\tau,~x>L,
\end{cases}
\end{equation}
for some $L>0$ and $\tau\in(0,1)$ such that (\ref{th1.2-cond}) in Theorem \ref{stability} holds. It is possible (by shift if necessary) to assume that
\[
\Phi _1(x,0)\geqslant \phi (x,0)~~\text{and}\ ~~\psi (x,0)\geqslant \psi (x,0), \ x\in\mathbb{ R}.
\]
By applying the comparison principle, we then obtain
\begin{equation}\label{eq6}
(\Phi _1,\Psi _1)(x+c_{-}^{*}(\alpha _1,\beta )t,t)\geqslant (\phi,\psi)(x,t), \  (x, t)\in \mathbb{R} \times \mathbb{R}^+.
\end{equation}
By Theorem\ref{stability}, we know that $(\phi,\psi)(x,t)$ is sufficiently close to the  $T$-periodic bistable traveling wave profile $\Gamma(z,t)=(\Phi,\Psi)(z,t)$ with $z=x+ct$. Then it follows that
\begin{equation}\label{eq66}
\Phi _1(x+c_{-}^{*}(\alpha _1,\beta )t,t)\geq \phi(x,t)\geq\Phi(x+ct,t)-\epsilon,
\end{equation}
for any  $\epsilon>0$. Let $\Phi_1(\xi, t)<1$ for the fixed points $(\xi,t), t>0$. By (\ref{eq66}), on the line $\xi=x+c_{-}^{*}(\alpha _1,\beta )t$, if $c>c_{-}^{*}(\alpha _1,\beta )$, then we have
$$
\Phi_1(\xi,t)\geq \Phi(\xi+(c-c_{-}^{*}(\alpha _1,\beta )t),t)-\epsilon\rightarrow 1-\epsilon  \   \text{as} \  \ t\rightarrow \infty.
$$
This is a contradiction since $\epsilon$ is arbitrary small. Thus $c\le c_{-}^{*}(\alpha _1,\beta )$.

The inequality in (\ref{speed2}) is a straightforward consequence of (\ref{speed1}) together with the proof process of Theorem \ref{existence}. Then the proof is complete.
\end{proof}

We now  establish the relationship between the bistable wave speed  and the wave speeds of upper/lower solutions of (\ref{wave sys}). The proofs of the following two theorems can be proceeded in the same way as in Theorem \ref{th3.1} and are omitted here.
%Relevant to the detailed process, we also refer the reader to Ma et al \cite{Ma_2019}.
\begin{theorem}\label{th3.2}
Assume that (\ref{wave sys}) has a nonnegative upper solution $(\overline{\Phi}(z,t),\overline{\Psi}(z,t))$ with speed $\bar{c},$, nondecreasing in $z$, $T$-periodic in $t$, and satisfying
\[
(\overline{\Phi},\overline{\Psi})(-\infty,t)<(1,1),\qquad (\overline{\Phi},\overline{\Psi})(\infty,t)\geqslant(1,1).
\]
Then the speed $c$ of the  bistable $T$-periodic traveling wave of (\ref{wave sys}) satisfies
\begin{equation}\label{upper_bound_of_c}
c\le \bar{c}.
\end{equation}
In particular, if  $\bar{c}<0$, then the bistable wave speed $c$ is negative.
\end{theorem}

\begin{theorem}\label{th3.3}
 Suppose that (\ref{wave sys}) has a nonnegative lower solution $(\underline{\Phi}(z,t),\underline{\Psi}(z,t))$ with  speed $\underline{c}$, nondecreasing in $z$, $T$-periodic in $t$ and satisfying
\[
(\underline{\Phi},\underline{\Psi})(-\infty,t)=(0,0)< (\underline{\Phi},\underline{\Psi})(\infty,t)\le(1,1).
\]
Then the speed $c$ of the bistable $T$-periodic traveling wave of (\ref{wave sys}) satisfies
\begin{equation}\label{lower_bound_of_c}
		c\geqslant\underline{c}.
\end{equation}
In particular, if  $\bar{c}>0$, then the bistable wave speed $c$ is positive.
\end{theorem}
Therefore, based on these two theorems, we could find explicit conditions for determining the sign of the bistable wave speed by seeking the formulas of upper/lower solutions of (\ref{wave sys})

\section{Result on the propagation direction}
In this section, we shall give some explicit  criteria to determine the sign of the bistable wave speed. To this end,
we first  discuss the  characteric equation of the bistable traveling wave  of system (\ref{wave sys})-(\ref{1.8}) near the equilibrium points $\textbf{o}$ and $\beta,$ which will be used to precisely construct upper solutions and lower solutions of (\ref{wave sys}).
\subsection{Eigenvalue  problem near $\textbf{o}$ and $\beta$}
Linearizing system (\ref{wave sys}) at $\textbf{o}$ yields
\begin{equation}\label{eigen_zero}
\begin{cases}
d_{1}(t)\left(\int_\mathbb{R }J_{1}(y)\Phi(z-y,t) dy -\Phi\right)-{c}{\Phi}_{z}-\Phi_{t}+{\Phi}[a_{1}(t)p(t)-b_{1}(t)q(t)]=0,\\
d_{2}(t)\left(\int_\mathbb{R} J_{2}(y)\Psi(z-y,t) dy -\Psi\right)-{c}{\Psi}_z-\Psi_{t}+a_{2}(t)p(t){\Phi}-b_{2}(t)q(t){\Psi}=0.\\
\end{cases}
\end{equation}
Let $(\Phi(z,t),\Psi(z,t))=(\varphi_{1}(t)e^{\mu z},\nu_{1}(t)e^{\mu z})$  solve (\ref{eigen_zero}). Then this leads to an eigenvalue problem
\begin{equation}\label{Separation of variables}
\begin{cases}
\varphi_{1}'(t)=\left[d_{1}(t)\left(\int_\mathbb{R} J_{1}(y)e^{-\mu y} dy -1\right)-{c}\mu+a_{1}(t)p(t)-b_{1}(t)q(t)\right]\varphi_{1}(t),\\
\nu_{1}'(t)=\left[d_{2}(t)\left(\int_\mathbb{R} J_{2}(y)e^{-\mu y} dy -1\right)-{c}\mu-b_{2}(t)q(t)\right]\nu_{1}(t)+a_{2}(t)p(t)\varphi_{1}(t),\\
\varphi_1(t)=\varphi_1(t+T),~\nu_{1}(t)=\nu_{1}(t+T).
\end{cases}
\end{equation}
By the first equation in (\ref{Separation of variables}), we have
\begin{equation}\label{rr}
I_1(\mu,c):=\int_0 ^T \left\{d_{1}(t)\left(\int_\mathbb{R} J_{1}(y)e^{-\mu y} dy -1\right)-{c}\mu+a_{1}(t)p(t)-b_{1}(t)q(t)\right\} dt=0.
\end{equation}
It is easy to check that $\frac{\partial^2 I_1}{\partial{\mu^2}}(\mu,c)\geq 0$ and $I_1(0,c)<0$, where (\ref{A2}) is used. Then $I_1(\mu,c)=0$ has only one positive root denoted by $\mu_1$ or $\mu_1(c)$.
If we let $\nu_{1}(t)=\rho_{1}(t)\varphi_{1}(t)$, then the second equation in (\ref{Separation of variables}) is changed into
\begin{equation}\label{tejie1}
\rho'_{1}(t)-\left( g_{1}(\mu_{}, t)-\dfrac{\varphi_{1}'(t)}{\varphi_{1}(t)}\right)\rho_{1}(t)=a_{2}(t)p(t),
\end{equation}
where
\begin{equation}\label{g1}
g_{1}(\mu,t)=d_{2}(t)\left(\int_\mathbb{R} J_{2}(y)e^{-\mu y} dy -1\right)-{c}\mu-b_{2}(t)q(t).
\end{equation}
Hence the linearized system has a solution such that
\begin{equation}\label{ap1}
(\Phi, \Psi) (z,t)\sim (\varphi_{1}(t), \  \rho_{1}(t)\varphi_{1}(t))e^{\mu_{1}z}  \  \text{as} \ z \rightarrow -\infty.
\end{equation}

However, let  $(\Phi(z,t),\Psi(z,t))=\left(\varphi_{1}(t)e^{\mu_1 z},\nu_{1}(t)e^{\mu_{2} z}+\rho_{1}(t)\varphi_{1}(t)e^{\mu_{1}z}\right)$ solve (\ref{eigen_zero}), where $\mu_2$ or $\mu_2(c)$ is another positive eigenvalue of (\ref{Separation of variables}) with $\mu_1\neq\mu_2$. Then we have that $\rho_{1}(t)$ still satisfies (\ref{tejie1}) and $\mu_2$ solves the following equation
\begin{equation}\label{b1}
h_1(\mu, c):=\int_0 ^T \Big\{d_{2}(t)\left(\int_\mathbb{R} J_{2}(y)e^{-\mu y} dy -1\right)-{c}\mu-b_{2}(t)q(t)\Big\} dt=0.
\end{equation}
Thus we have a solution with
\begin{equation}\label{ap2}
 (\Phi, \Psi) (z,t)\sim \left(\varphi_{1}(t)e^{\mu_1 z}, \ \nu_{1}(t)e^{\mu_{2} z}+\rho_{1}(t)\varphi_{1}(t)e^{\mu_{1}z}\right) \  \text{as} \ z \rightarrow -\infty.
 \end{equation}

Next linearizing  system (\ref{wave sys}) around the equilibrium $\beta$, we have
\begin{equation}\label{eigen_one}
\begin{cases}
d_{1}(t)\left(\int_\mathbb{R} J_{1}(y)\Phi(z-y,t) dy -\Phi\right)-{c}{\Phi}_{z}-\Phi_{t}-a_{1}(t)p(t){\Phi}+b_{1}(t)q(t)\Psi=0,\\
d_{2}(t)\left(\int_\mathbb{R} J_{2}(y)\Psi(z-y,t) dy -\Psi\right)-{c}{\Psi}_z-\Psi_{t}+[b_{2}(t)q(t)-a_{2}(t)p(t)]{\Psi}=0.\\
\end{cases}
\end{equation}
Let  (\ref{eigen_one}) have solutions in the form of $(\Phi(z,t),\Psi(z,t))=(\varphi_{2}(t)e^{-\mu z},\nu_{2}(t)e^{-\mu z})$. Then the eigenvalue problem associated with it is
\begin{equation}\label{Separation of variables2}
\begin{cases}
\nu_{2}'(t)=\left[d_{2}(t)\left(\int_\mathbb{R} J_{2}(y)e^{\mu y} dy -1\right)+{c}\mu+b_{2}(t)q(t)-a_{2}(t)p(t)\right]\nu_{2}(t),\\
\varphi_{2}'(t)=\left[d_{1}(t)\left(\int_\mathbb{R} J_{1}(y)e^{\mu y} dy -1\right)+{c}\mu-a_{1}(t)p(t)\right]\varphi_{2}(t)+b_{1}(t)q(t)\nu_{2}(t),\\
\varphi_2(t)=\varphi_2(t+T),~\nu_{2}(t)=\nu_{2}(t+T).
\end{cases}
\end{equation}
Integrating the first equation over the interval $[0,T]$ produces
\begin{equation}\label{ee}
I_{2}(\mu,c):=\int_0 ^T\left\{ d_{2}(t)\left(\int_\mathbb{R} J_{2}(y)e^{\mu y} dy -1\right)+{c}\mu+b_{2}(t)q(t)-a_{2}(t)p(t) \right\}dt=0.
\end{equation}
By a simple computation, we have $\frac{\partial^2I_2}{\partial \mu^2}(\mu,c)\geq0$. Furthermore,  (\ref{A2}) implies $I_{2}(0,c)<0$. Hence (\ref{ee}) has only one positive root denoted by $\mu_{4}$ or $\mu_{4}(c)$. If assume  $\varphi_{2}(t)=\rho_{2}(t)\nu_2(t)$, then the second equation in (\ref{Separation of variables2}) yields
\begin{equation}\label{tejie2}
\rho'_{2}(t)-\left( g_{2}(\mu_{},t)-\dfrac{\nu_{2}'(t)}{\nu_{2}(t)}\right)\rho_{2}(t)= b_{1}(t)q(t),
\end{equation}
where
\begin{equation}\label{g2}
g_{2}(\mu,t)=d_{1}(t)\left(\int_\mathbb{R} J_{1}(y)e^{\mu y} dy -1\right)+{c}\mu-a_{1}(t)p(t).
\end{equation}
This means that the solution of   system (\ref{wave sys}) may have a behavior
\begin{equation}\label{ap3}
(\Phi, \Psi)(z,t)\sim (1,1)+( \rho_{2}(t)\nu_2(t), \  \nu_{2}(t))e^{-\mu_{4}z} \  \text{as} \ z\rightarrow\infty.
\end{equation}

 However, we suppose that  $(\Phi(z,t), \Psi(z,t))=\left(\varphi_{2}(t)e^{-\mu_{3} z}+\rho_{2}(t)\nu_{2}(t)e^{-\mu_{4}z}, \nu_{2}(t)e^{-\mu_{4} z}\right)$ is a solution of  (\ref{Separation of variables2}), where $\mu_3$ or $\mu_3(c)$ is another positive eigenvalue of (\ref{Separation of variables2}) and $\mu_3\neq\mu_4$. Then it is easy to verify that $\rho_{2}$ still solves (\ref{tejie2}) and $\mu_3$ satisfies
\begin{equation}\label{h2}
h_2(\mu,c):=\int_0 ^T \left\{d_{1}(t)\left(\int_\mathbb{R} J_{1}(y)e^{\mu y} dy -1\right)+{c}\mu-a_{1}(t)p(t)\right\}dt=0.
\end{equation}
Then the asymptotical behavior of the bistable traveling wave to system (\ref{wave sys}) near $\beta$ may become
\begin{equation}\label{ap4}
(\Phi, \Psi)(z,t)\sim \left(1+\varphi_{2}(t)e^{-\mu_{3} z}+\rho_{2}(t)\nu_{2}(t)e^{-\mu_{4}z}, \  1+\nu_{2}(t)e^{-\mu_{4} z}\right) \  \text{as} \  z\rightarrow\infty.
\end{equation}
 The positivity of   $\varphi_1(t)$ and $\nu_2(t)$  is easy to verify, which  shall be used to construct upper/lower solutions  of (\ref{wave sys}).
 %The following lemma is obviously true and plays an important role in proving the uniqueness of the bistable wave speed.
% \begin{lemma}\label{asym}
%For $c$ in (\ref{speed1}), the decay rates $\mu_1(c)$ and $\mu_2(c)$ are increasing in $c$; while the decay rates $\mu_1(c)$ and $\mu_2(c)$  are decreasing in $c$.
%
%\end{lemma}

\subsection{Explicit condition for propagation direction}
According to the results in the last subsection and Theorems \ref{th3.2} and \ref{th3.3}, under condition (\ref{A2}),
we shall derive explicit conditions for determining the sign of the bistable wave speed by constructing upper/lower solutions. By (\ref{def}), it is well known that the bistable traveling wave with positive (negative) speed propagates to the left (right). We begin with introducing two functions
\begin{equation}\label{y1}
Y_{1}(\mu_1(c),t)=d_{1}(t)\left(\int_{\mathbb{R}}J_{1}(y)e^{\mu_{1}(c)y}dy-1\right)-d_{2}(t)\left(\int_{\mathbb{R}}J_{2}(y)e^{\mu_{1}(c)y}dy-1\right)
\end{equation}
and
\begin{equation}\label{y2}
Y_{2}(\mu_1(c),t)=-d_{1}(t)\int_{0} ^{\infty}  J_{1}(y)(1+e^{\mu_{1}(c)y})(e^{\frac{1}{2}\mu_{1}(c)y}-e^{-\frac{1}{2}\mu_{1}(c)y})^2 dy.
\end{equation}
\begin{theorem}\label{TH1}
Assume that the T-periodic coefficients $d_i(t),  a_{i}(t), b_{i}(t) (i=1,2)$ satisfy
\begin{equation}\label{TH1-cond1}
0<\dfrac{Y_{1}(\mu_1(0),t)+a_{1}(t)p(t)-b_{1}(t)q(t)+b_{2}(t)q(t)}{a_{2}(t)p(t)}<\dfrac{a_{ 1}(t)p(t)+Y_{2}(\mu_1(0),t)}{a_{1}(t)p(t)}, \  t\in [0,T],
\end{equation}
where $ p(t), q(t)$ are defined in (\ref{p(t)}) and $\mu_1(0)$ solves
\begin{equation}\label{rr1}
I_1(\mu_1(0),0):=\int_0 ^T \left\{d_{1}(t)\left(\int_\mathbb{R} J_{1}(y)e^{-\mu_1(0) y} dy -1\right)+a_{1}(t)p(t)-b_{1}(t)q(t)\right\} dt=0.
\end{equation}
Then the bistable wave speed of system (\ref{wave sys}) is positive.
\end{theorem}
\begin{proof}
By (\ref{TH1-cond1}) and $Y_{2}(\mu_1(0), t)<0$, there exists a real number $k_{1}$ satisfying
\begin{equation}\label{TH1-cond2}
\dfrac{Y_{1}(\mu_1(0),t)+a_{1}(t)p(t)-b_{1}(t)q(t)+b_{2}(t)q(t)}{a_{2}(t)p(t)}<k_{1}<\dfrac{a_{1}(t)p(t)+Y_{2}(\mu_1(0),t)}{a_{1}(t)p(t)}<1.
\end{equation}
Define a pair of functions $(\underline{\Phi},\underline{\Psi})$ by
\begin{equation}\label{lower01}
	\underline{\Phi}(z,t) = \dfrac{k_1 \varphi_{1}(t)}{\varphi_{1}(t)+e^{-\mu_{1}(\underline{c})z}},~~
\underline{\Psi}(z,t)=\dfrac{1}{k_{1}}\underline{\Phi}
\end{equation}
with speed $ 0<\underline{c}\ll 1$. If $(\underline{\Phi},\underline{\Psi})$ can be proved to be a lower solution of (\ref{wave sys}), then Theorem \ref{th3.3} implies the desired result.

Indeed, substituting $(\underline{\Phi},\underline{\Psi})$ into (\ref{wave sys}), from the first equation it follows that
\begin{equation*}
\begin{split}
&d_{1}(t)\left(\int_\mathbb{R} J_{1}(y)\underline{\Phi}(z-y,t) dy -\underline{\Phi}\right)-\underline{c}\underline{\Phi}_{z}-\underline{\Phi}_{t}+\underline{\Phi}\left[a_{1}(t)p(t)(1-\underline{\Phi})
-b_{1}(t)q(t)(1-\underline{\Psi})\right] \\
&=\underline{\Phi}(1-\dfrac{\underline{\Phi}}{k_{1}})
\bigg\{\dfrac{d_{1}(t)\left(\int_R J_{1}(y)\underline{\Phi}(z-y,t) dy -\underline{\Phi}\right)}{\underline{\Phi}(1-\dfrac{\underline{\Phi}}{k_{1}})}-\underline{c}\mu_{1}(\underline{c})-\dfrac{\varphi_{1}'(t)}{\varphi_{1}(t)} +H_{1}(z,t)
\bigg\}\\
&\stackrel{def}{=}\Lambda_1, \\
\end{split}
\end{equation*}
where
\[
H_{1}(z,t)=\dfrac{a_{1}(t)p(t)(1-\underline{\Phi})-b_{1}(t)q(t)(1-\dfrac{\underline{\Phi}}{k_{1}})}{1-\dfrac{\underline{\Phi}}{k_1}}.
\]
Now applying  the first equation in (\ref{Separation of variables}), we have
\begin{equation*}
\begin{split}
\Lambda_1&=\dfrac{\underline{\Phi}^{2}}{k_{1}}(1-\dfrac{\underline{\Phi}}{k_{1}})
\bigg\{\dfrac{d_{1}(t)\int_\mathbb{R} J_{1}(y)\left(\underline{\Phi}(z-y,t) -\underline{\Phi}-(e^{-\mu_{1}(\underline{c})y}-1)\underline{\Phi}(1-\dfrac{\underline{\Phi}}{k_{1}})\right)dy}{\dfrac{\underline{\Phi}^{2}}{k_{1}}(1-\dfrac{\underline{\Phi}}{k_{1}})}\\
&~~~~~~~~~~~~~~~~~~~+\dfrac{a_{1}(t)p(t)(1-k_{1})}{1-\dfrac{\underline{\Phi}}{k_{1}}}
\bigg\} \\
&=\dfrac{\underline{\Phi}^{2}}{k_{1}}(1-\dfrac{\underline{\Phi}}{k_{1}})
\bigg\{ d_{1}(t)\int_\mathbb{R} J_{1}(y)S(\mu_{1}(\underline{c}),z,y,t)(2-e^{\mu_{1}(\underline{c}) y}-e^{-\mu_{1}(\underline{c}) y})dy+\dfrac{a_{1}(t)p(t)(1-k_{1})}{1-\dfrac{\underline{\Phi}}{k_{1}}} \bigg\}, \\
\end{split}
\end{equation*}
where $ S(\mu_{1}(\underline{c}),z,y,t)=\frac{\varphi_{1}(t)e^{\mu_{1}(\underline{c})z}+1}{\varphi_{1}(t)e^{\mu_{1}(\underline{c})z}+e^{\mu_{1}y}}$.  It is easy to verify
\begin{equation}\label{SS}
\left\{
\begin{array}{lr}
e^{-\mu_1(\underline{c}) y} <S(\mu_{1}(\underline{c}),z,y,t)< 1 \   \text{for} \  y\geq0, \\[2mm]
1< S(\mu_{1}(\underline{c}),z,y,t)< e^{-\mu_1(\underline{c}) y} \  \text{for} \  y\leq0.
\end{array}
\text{$(z,t) \in \mathbb{R} \times \mathbb{R}_+$}
\right.
\end{equation}
By this,  the first equation in (\ref{Separation of variables}) and (\ref{TH1-cond2}), we have
\begin{equation}\label{bb1}
\begin{split}
\Lambda_1&\geqslant\dfrac{\underline{\Phi}^{2}}{k_{1}}(1-\dfrac{\underline{\Phi}}{k_{1}})\left[Y_2(\mu_1(\underline{c}),t)+a_{1}(t)p(t)-a_{1}(t)p(t)k_1\right]\\
&\rightarrow \dfrac{\underline{\Phi}^{2}}{k_{1}}(1-\dfrac{\underline{\Phi}}{k_{1}})\left[Y_2(\mu_1(0),t)+a_{1}(t)p(t)-a_{1}(t)p(t)k_1\right]>0 \ \text{as} \ \underline{c} \rightarrow 0^+.
\end{split}
\end{equation}
For the second equation in (\ref{wave sys}), by (\ref{SS}) and  (\ref{TH1-cond2}), we have
\begin{equation}\label{bb2}
\begin{split}
&d_{2}(t)\left(\int_\mathbb{R} J_{2}(y)\underline{\Psi}(z-y,t) dy -\underline{\Psi}\right)-{\underline{c}}\underline{\Psi}_{z}-\underline{\Psi}_{t}+(1-\underline{\Psi})\left[a_{2}(t)p(t)\underline{\Phi}-b_{2}(t)q(t)\underline{\Psi}\right]\\
=&\dfrac{\underline{\Phi}}{k_{1}}(1-\dfrac{\underline{\Phi}}{k_{1}}) \bigg\{d_{2}(t)\int_\mathbb{R} J_{2}(y)S(\mu_{1}(\underline{c}),z,y,t)(1-e^{\mu_{1}(\underline{c})y})dy-\underline{c}\mu_{1}(\underline{c})-\dfrac{\varphi'_{1}(t)}{\varphi_{1}(t)}\\
&~~~~~~~~~~~~~~~~+a_{2}(t)p(t)k_{1}-b_{2}(t)q(t) \bigg\}\\
>&\dfrac{\underline{\Phi}}{k_{1}}(1-\dfrac{\underline{\Phi}}{k_{1}}) \left\{\int_\mathbb{R} J_{2}(y)(1-e^{\mu_{1}(\underline{c})y})dy-\underline{c}\mu_{1}(\underline{c})-\dfrac{\varphi'_{1}(t)}{\varphi_{1}(t)}+a_{2}(t)p(t)k_{1}-b_{2}(t)q(t) \right\}\\
=& \dfrac{\underline{\Phi}}{k_{1}}(1-\dfrac{\underline{\Phi}}{k_{1}}) \bigg\{\int_\mathbb{R} J_{2}(y)(1-e^{\mu_{1}(\underline{c})y})dy-d_{1}(t)\left(\int_\mathbb{R } J_{1}(y)e^{-\mu_{1}(\underline{c})y}dy-1\right)\\
&~~~~~~~~~~~~~~~~-a_{1}(t)p(t)+b_{1}(t)q(t)-b_{2}(t)q(t)+a_{2}(t)p(t)k_{1} \bigg\}\\
=&\dfrac{\underline{\Phi}}{k_{1}}(1-\dfrac{\underline{\Phi}}{k_{1}}) \bigg\{a_{2}(t)p(t)k_{1}-[Y_1(\mu_1(\underline{c}),t)+a_{1}(t)p(t)-b_{1}(t)q(t)+b_{2}(t)q(t)]\bigg\}\\
\rightarrow& \dfrac{\underline{\Phi}}{k_{1}}(1-\dfrac{\underline{\Phi}}{k_{1}}) \bigg\{a_{2}(t)p(t)k_{1}-[Y_1(\mu_1(0),t)+a_{1}(t)p(t)-b_{1}(t)q(t)+b_{2}(t)q(t)]\bigg\}>0   \ \text{as} \ \underline{c} \rightarrow 0^+.
\end{split}
\end{equation}
Thus, by (\ref{bb1}) and (\ref{bb2}),  as $\underline{c}$ is sufficiently close to $0$,  $(\underline{\Phi},\underline{\Psi})$ is a lower solution of (\ref{wave sys}), and the proof is complete.
\end{proof}

In order to obtain conditions for the negative wave speed, we next construct an explicit upper solution which possesses two piecewise continuous components.
\begin{theorem}\label{TH2}
		Let
		\begin{equation}\label{con1}
			Y_{3}(\mu_1(c),t)=-d_{1}(t)\int_{-\infty} ^{0} J_{1}(y)(1+e^{\mu_{1}(c)y})(e^{\frac{1}{2}\mu_{1}(c)y}-e^{-\frac{1}{2}\mu_{1}(c)y})^2 dy.
		\end{equation}
and
        \begin{equation}\label{F}
            F\left(\mu_{1}(c),s_{0},t\right)=d_{2}(t)\int_\mathbb{R} J_{2}(y)\left(2+\dfrac{(1-s_{0})(1-e^{\mu_{1}(c)y})}{s_{0}+(1-s_{0})e^{\mu_{1}(c)y}}\right)(1-e^{\mu_{1}(c)y})dy,
        \end{equation}
		where $s_{0}\in(0,1)$ is a constant.
 %dependent of system parameters.
 Suppose that there exists $s_0$ such that the T-periodic coefficients $ d_i(t)$, $a_{i}(t)$, $ b_{i}(t)$ $(i=1,2)$ satisfy
		\begin{equation}\label{TH2-cond1}
            \max\left\{\dfrac{a_{2}(t)p(t)}{b_{2}(t)q(t)},1\right\}
            <\min\left\{\dfrac{1-\frac{a_{1}(t)p(t)}{b_{1}(t)q(t)}(1-s_{0})}{s_{0}},\dfrac{-Y_{3}(\mu_1(0),t)}{b_{1}(t)q(t)}\right\} , \  t\in [0,T],
        \end{equation}
and
        \begin{equation}\label{TH2-cond2}
		      F\left(\mu_{1}(0),s_{0},t\right)< Y_1(\mu_1(0), t)+a_1(t)p(t)-b_1(t)q(t), \  t\in [0,T],
        \end{equation}
		where  $ p(t), q(t)$, $Y_1$ and $\mu_1(0)$  are defined in (\ref{p(t)}), (\ref{y1}) and (\ref{rr}), respectively.  Then the bistable wave speed  of system (\ref{wave sys}) is negative.
	\end{theorem}
	\begin{proof}
		By condition (\ref{TH2-cond1}), we can take a constant $k_{4}$ such that
		\begin{equation}\label{TH2-cond2a}
			\max\left\{\dfrac{a_{2}(t)p(t)}{b_{2}(t)q(t)},1\right\}<k_{4}<\min\left\{\dfrac{1-\frac{a_{1}(t)p(t)}{b_{1}(t)q(t)}(1-s_{0})}{s_{0}},\dfrac{-Y_{3}(\mu_1(0),t)}{b_{1}(t)q(t)}\right\}, \  t\in [0,T].
		\end{equation}
			Now let a pair of nondecreasing functions be defined by
		\begin{equation*}
			\overline{\Phi} (z,t) =\left\{
				\begin{array}{lr}
					s_{0}, & z< z_{1}(t),\\
					\dfrac{\varphi_{1}(t)}{\varphi_{1}(t)+e^{-\mu_{1}(\bar{c})z}}, &z\geq z_{1}(t)
				\end{array}
			\right. ,~~~~~~~~
			\overline{\Psi}(z,t)=\left\{
			\begin{array}{lr}
				k_{4}\overline{\Phi}, & z\le z_{2}(t),\\
				1, &z>z_{2}(t)
			\end{array}
			\right.
		\end{equation*}
		with $-1\ll\overline{c}<0$, $z_{1}(t)<z_{2}(t)$ such that $ \overline{\Phi}(z_1(t),t)=s_{0} $, $k_{4}\overline{\Phi}(z_2(t),t)=1$. For  $z>z_{1}(t) $, it is easy to check that
		\begin{equation}\label{shiji-1}
			\overline{\Phi}_{t}=\dfrac{\varphi_{1}'(t)}{\varphi_{1}(t)}\overline{\Phi}(1-\overline{\Phi}),
~~~\overline{\Phi}_{z}=\mu_{1}(\overline{c})\overline{\Phi}(1-\overline{\Phi}),   \   z>z_{1}(t).
		\end{equation}
		
		If we prove that $(\overline{\Phi},\overline{\Psi})$ is an upper solution of system (\ref{wave sys}), then we can get the desired result. To this end, we substitute $(\overline{\Phi},\overline{\Psi})$ into (\ref{wave sys}).
		When $z< z_1(t)$, by using the first equation of (\ref{Separation of variables}) and (\ref{TH2-cond2a}), the first equation in (\ref{wave sys}) becomes
		\begin{eqnarray*}
			\begin{aligned}
				&d_{1}(t)\left(\int_\mathbb{R} J_{1}(y)\overline{\Phi}(z-y,t) dy -\overline{\Phi}\right)-\overline{c}\overline{\Phi}_{z}-\overline{\Phi}_{t}+\overline{\Phi}\left[a_{1}(t)p(t)(1-\overline{\Phi})
				-b_{1}(t)q(t)(1-\overline{\Psi})\right] \\
				=&s_{0}\left[a_{1}(t)p(t)(1-s_{0})
				-b_{1}(t)q(t)(1-k_4s_{0})\right]<0,
			\end{aligned}
		\end{eqnarray*}
and the second equation in (\ref{wave sys}) becomes
		\begin{equation*}
			\begin{aligned}
				&d_{2}(t)\left(\int_\mathbb{R} J_{2}(y)\overline{\Psi}(z-y,t) dy -\overline{\Psi}\right)-\overline{c}\overline{\Psi}_{z}-\overline{\Psi}_{t}+(1-\overline{\Psi})\left[a_{2}(t)p(t)\overline{\Phi}-b_{2}(t)q(t)\overline{\Psi}\,\right]\\
				=&(1-k_{4}s_{0})s_{0}\left[a_{2}(t)p(t)-b_{2}(t)q(t)k_{4}\right]<0
			\end{aligned}
		\end{equation*}
		When $z_{1}(t)\leqslant z\leqslant z_2(t),$   the first equation in (\ref{wave sys}) becomes
		\begin{eqnarray*}
			\begin{aligned}
				&d_{1}(t)\left(\int_\mathbb{R} J_{1}(y)\overline{\Phi}(z-y,t) dy -\overline{\Phi}\right)-\overline{c}\overline{\Phi}_{z}-\overline{\Phi}_{t}+\overline{\Phi}\left[a_{1}(t)p(t)(1-\overline{\Phi})
				-b_{1}(t)q(t)(1-\overline{\Psi})\right] \\
				=&\overline{\Phi}^{2}(1-\overline{\Phi})
				\bigg\{\dfrac{d_{1}(t)\left(\int_\mathbb{R} J_{1}(y)\overline{\Phi}(z-y,t) dy -\overline{\Phi}\right)}{\overline{\Phi}^{2}(1-\overline{\Phi})}-\dfrac{d_{1}(t)\left(\int_\mathbb{R }J_{1}(y)e^{-\mu_{1}(\overline{c})y}dy-1\right)}{\overline{\Phi}}\\
				&~~~~~~~~~~~~~~~~~+\dfrac{k_{4}b_{1}(t)q(t)-b_{1}(t)q(t)}{1-\overline{\Phi}}\bigg\}\\
				=&\overline{\Phi}^{2}(1-\overline{\Phi})
				\bigg\{ d_{1}(t)\int_\mathbb{R }J_{1}(y)S(\mu_{1}(\overline{c}),z,y,t)(2-e^{-\mu_{1}(\overline{c}) y}-e^{\mu_{1} y})dy+\dfrac{b_{1}(t)q(t)(k_{4}-1)}{1-\overline{\Phi}}  \bigg\}\\
				<&\overline{\Phi}^{2}(1-\overline{\Phi})
				\bigg\{ d_{1}(t)\int_\mathbb{R} J_{1}(y,t)S(\mu_{1}(\overline{c}),z,y,t)(2-e^{\mu_{1}(\overline{c}) y}-e^{-\mu_{1}(\overline{c}) y})dy+b_{1}(t)q(t)k_{4}       \bigg\} \\
				<& \overline{\Phi}^{2}(1-\overline{\Phi}) \bigg\{Y_{3}(\mu_1(\overline{c}),t)+b_{1}(t)q(t)k_{4}   \bigg\}\\
				\rightarrow& \overline{\Phi}^{2}(1-\overline{\Phi}) \bigg\{Y_{3}(\mu_1(0),t)+b_{1}(t)q(t)k_{4}   \bigg\}<0 \   \  \text{as} \  \  \overline{c}\rightarrow 0^-.
			\end{aligned}
		\end{eqnarray*}
To proceed, we give the estimates $ S(\mu_{1}(\bar{c}),z,y,t) $ for $z\in[z_{1}(t), z_{2}(t)]$ and $t>0$ by
		\begin{equation}\label{SSS}
			\left\{
			\begin{array}{lr}
				&1+\dfrac{1-e^{\mu_{1}(\bar{c}) y}}{\varphi_{1}(t)e^{\mu_{1}(\bar{c})z_{1}(t)}+e^{\mu_{1}(\bar{c})y}}
				\leqslant S(\mu_{1}(\bar{c}),z,y,t)\leqslant 1+\dfrac{1-e^{\mu_{1}(\bar{c}) y}}{\varphi_{1}(t)e^{\mu_{1}(\bar{c})z_{2}(t)}+e^{\mu_{1}(\bar{c})y}}
				 \   \text{for} \  y\geq0, \\[8mm]
				&1+\dfrac{1-e^{\mu_{1}(\bar{c}) y}}{\varphi_{1}(t)e^{\mu_{1}(\bar{c})z_{2}(t)}+e^{\mu_{1}(\bar{c})y}}
				\leqslant S(\mu_{1}(\bar{c}),z,y,t)\leqslant 1+\dfrac{1-e^{\mu_{1}(\bar{c}) y}}{\varphi_{1}(t)e^{\mu_{1}(\bar{c})z_{1}(t)}+e^{\mu_{1}(\bar{c})y}} \  \text{for} \  y\leq0.
			\end{array}
			\right.
		\end{equation}
		Thus for the second equation, we have
		\begin{eqnarray*}
			\begin{aligned}
				&d_{2}(t)\left(\int_\mathbb{R} J_{2}(y)\overline{\Psi}(z-y,t) dy -\overline{\Psi}\right)-\overline{c}\overline{\Psi}_{z}-\overline{\Psi}_{t}+(1-\overline{\Psi})\left[a_{2}(t)p(t)\overline{\Phi}-b_{2}(t)q(t)\overline{\Psi}\right] \\
				=&d_{2}(t)\left(\int_\mathbb{R} J_{2}(y)k_{4}\overline{\Phi}(z-y,t) dy -k_{4}\overline{\Phi}\right)-\overline{c}k_{4}\overline{\Phi}_{z}-k_{4}\overline{\Phi}_{t}+\overline{\Phi}(1-k_{4}\overline{\Phi})\left[a_{2}(t)p(t)
				-b_{2}(t)q(t)k_{4}\right]\\
				<& \overline{\Phi}(1-\overline{\Phi})\bigg\{\dfrac{k_{4}d_{2}(t)\left(\int_\mathbb{R} J_{2}(y)\overline{\Phi}(z-y,t) dy -\overline{\Phi}\right)}{\overline{\Phi}(1-\overline{\Phi})}-k_{4}\overline{c}\mu_{1}(\overline{c})-k_{4}\dfrac{\varphi_{1}'(t)}{\varphi_{1}(t)}
				\bigg\} \\
				=& \overline{\Phi}(1-\overline{\Phi})k_{4}\bigg\{ d_{2}(t)\int_\mathbb{R} J_{2}(y)S(\mu_{1}(\bar{c}),z,y,t)(1-e^{\mu_{1}(\overline{c})y})dy -\overline{c}\,\mu_{1}(\overline{c})-\dfrac{\varphi_{1}'(t)}{\varphi_{1}(t)}\bigg\} \\
				=& \overline{\Phi}(1-\overline{\Phi})k_{4}\bigg\{-Y_1(\mu_1(\overline{c}), t)-a_1(t)p(t)+b_1(t)q(t)\\
&~~~~~~~~~~~~~~~~+d_{2}(t)\int_\mathbb{R} J_{2}(y)\left(S(\mu_{1}(\bar{c}),z,y,t)+1\right)(1-e^{\mu_{1}(\overline{c})y})dy\bigg\} \\
				\leqslant&
				\overline{\Phi}(1-\overline{\Phi})k_{4}\bigg\{-Y_1(\mu_1(\overline{c}), t)-a_1(t)p(t)+b_1(t)q(t)\\
&~~~~~~~~~~~~~~~~+d_{2}(t)\int_\mathbb{R} J_{2}(y)\left(S(\mu_{1}(\bar{c}),z_{1}(t),y,t)+1\right)(1-e^{\mu_{1}(\overline{c})y})dy\bigg\}\\
				=&
				\overline{\Phi}(1-\overline{\Phi})k_{4}\bigg\{-Y_1(\mu_1(\overline{c}), t)-a_1(t)p(t)+b_1(t)q(t)+F\left(\mu_{1}(\overline{c}),s_{0},t\right)\bigg\}\\
				\rightarrow& \overline{\Phi}(1-\overline{\Phi})k_{4}\bigg\{-Y_1(\mu_1(0), t)-a_1(t)p(t)+b_1(t)q(t)+F\left(\mu_{1}(0),s_{0},t\right)\bigg\}<0 \  \text{as} \  \  \overline{c}\rightarrow 0^-.
			\end{aligned}
		\end{eqnarray*}
		where $F\left(\mu_{1}(\overline{c}),s_{0},t\right)$ is defined in (\ref{F}).
		For $z>z_2(t)$, we have $\overline{\Psi}=1$. Then from the first equation in (\ref{wave sys}) it follows that
		\begin{eqnarray*}
			\begin{aligned}
				&d_{1}(t)\left(\int_\mathbb{R} J_{1}(y)\overline{\Phi}(z-y,t) dy -\overline{\Phi}\right)-\overline{c}\overline{\Phi}_{z}-\overline{\Phi}_{t}+\overline{\Phi}\left[a_{1}(t)p(t)(1-\overline{\Phi})
				-b_{1}(t)q(t)(1-\overline{\Psi})\right] \\
				=&\overline{\Phi}(1-\overline{\Phi})
				\bigg\{\dfrac{d_{1}(t)\left(\int_\mathbb{R} J_{1}(y)\underline{\Phi}(z-y,t) dy -\overline{\Phi}\right)}{\overline{\Phi}(1-\overline{\Phi})}-\overline{c}\mu_{1}(\overline{c})-\dfrac{\varphi_{1}'(t)}{\varphi_{1}(t)}+a_{1}(t)p(t)\bigg\}\\
				=&\overline{\Phi}(1-\overline{\Phi})
				\bigg\{\dfrac{d_{1}(t)\left(\int_\mathbb{R} J_{1}(y)\overline{\Phi}(z-y,t) dy -\overline{\Phi}\right)}{\overline{\Phi}(1-\overline{\Phi})}-d_{1}(t)\left(\int_\mathbb{R } J_{1}(y)e^{-\mu_{1}(\overline{c})y}dy-1\right)+b_{1}(t)q(t)\bigg\}\\
				=&\overline{\Phi}^{2}(1-\overline{\Phi})
				\bigg\{\dfrac{d_{1}(t)\int_\mathbb{R} J_{1}(y)\left[\overline{\Phi}(z-y,t) -\overline{\Phi}-(e^{-\mu_{1}(\overline{c})y}-1)\overline{\Phi}(1-\overline{\Phi})\right]dy}{\overline{\Phi}^{2}(1-\overline{\Phi})} +\dfrac{b_{1}(t)q(t)}{\overline{\Phi}}\bigg\}\\
				<&\overline{\Phi}^{2}(1-\overline{\Phi})
				\bigg\{d_{1}(t)\int_\mathbb{R} J_{1}(y)S(\mu_{1}(\bar{c}),z,y,t)(2-e^{-\mu_{1}(\overline{c}) y}-e^{\mu_{1}(\overline{c}) y})dy +b_{1}(t)q(t)k_{4}\bigg\}\\
				<& \overline{\Phi}^{2}(1-\overline{\Phi}) \bigg\{Y_{3}(\mu_1(\overline{c}),t)+b_{1}(t)q(t)k_{4}   \bigg\}\\
				\rightarrow&  \overline{\Phi}^{2}(1-\overline{\Phi}) \bigg\{Y_{3}(\mu_1(0),t)+b_{1}(t)q(t)k_{4}   \bigg\} <0 \  \text{as} \  \  \overline{c}\rightarrow 0^-.
			\end{aligned}
		\end{eqnarray*}
		and the second equation in (\ref{wave sys}) becomes
		\begin{eqnarray}
				&d_{2}(t)\left(\int_\mathbb{R} J_{2}(y)\overline{\Psi}(z-y,t) dy -\overline{\Psi}\right)-\overline{c}\overline{\Psi}_{z}-\overline{\Psi}_{t}+(1-\overline{\Psi})\left[a_{2}(t)p(t)\overline{\Phi}-b_{2}(t)q(t)\overline{\Psi}\,\right]=0.\nonumber\\
		\end{eqnarray}
		Hence $(\overline{\Phi}, \overline{\Psi})$ is an upper solution of  (\ref{wave sys}) as $\overline{c}$ is sufficiently close to $0$. The proof is complete.
\end{proof}

\section{Examples and simulations}
In this section we present two examples to demonstrate the results of Theorems \ref{TH1} and \ref{TH2} when the condition (A2) is satisfied, but the condition (\ref{1.12a}) is not.

 Theorem \ref{TH1} indicates that the condition (\ref{TH1-cond1}) can guarantee that the bistable wave  speed is positive, that is, the bistable traveling wave connecting $(0, q(t) )$ to $(p(t),0)$ propagates to the left,  which means the stable state ($ p(t) , 0 $) wins the competition, and thus the species $ u $ will tend to the periodic state $ p(t) $ and the species $ v $ will tend to  extinction as time increases.

 Theorem \ref{TH2} shows that if there exists a constant $s_0$ such that both (\ref{TH2-cond1}) and (\ref{TH2-cond2}) are satisfied, the bistable wave  speed is negative. Therefore, with the increase of the time, the species $ u $ will become extinct and the species $ v $ will approach the periodic state $ q(t)$.

 In the two examples, the kernel functions $J_i,i=1,2$ are taken as
 \begin{equation}\label{J}
        J_{1}(y)=J_{2}(y)=\dfrac{1}{\sqrt{2\pi}}e^{-\frac{y^2}{2}},\quad -\infty<y<+\infty.
    \end{equation}
 It is easy to verify that $J_i,i=1,2$  satisfy (\textbf{A1})-(\textbf{A3}). The simulation is to directly integrate the full system  (\ref{original_model}) with the initial data
   \begin{equation}
	u(x,0)=\dfrac{p_{0}}{1+e^{-x}},\quad v(x,0)=\dfrac{q_{0}}{1+e^{x}},
\end{equation}
where $ p_{0} $ and $ q_{0} $ are defined in (\ref{p(t)}).

 In Example 1,  the coefficient functions are taken as
\begin{equation}\label{coeff_1}
	\begin{aligned}
		&d_{1}(t)=10, \quad r_{1}(t)=3.5, \quad a_{1}(t)=3\sin(2t)+5, \quad \;\; b_{1}(t)=3\sin(2t)+10;\\
		&d_{2}(t)=15, \quad r_{2}(t)=3, \quad\;\;\, a_{2}(t)=3\cos(2t)+15, \quad b_{2}(t)=3\cos(2t)+8.\\
	\end{aligned}
\end{equation}
Then it is easy to check that (\ref{A2})(not (\ref{1.12a})) and the condition (\ref{TH1-cond1}) in Theorem \ref{TH1} are satisfied. The propagation behavior of the $\pi$-periodic bistable traveling wave  is displayed in Fig 1.
\begin{figure}[H]
	\centering
	\includegraphics[scale=0.60]{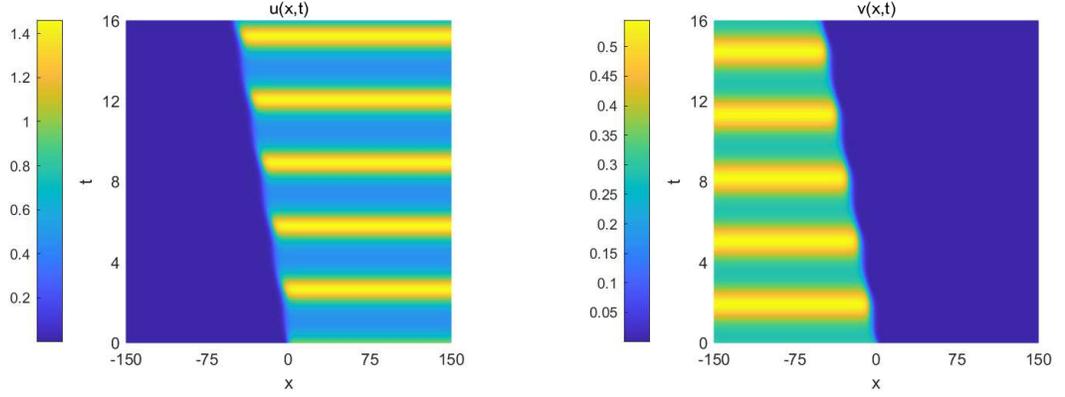}
    \captionsetup{font={small},labelfont=bf,labelformat={default},name={Fig.},labelsep=space}
	\caption{Top views of the $\pi$-periodic bistable wave of (\ref{original_model}). The coefficient functions are the same as in (\ref{coeff_1}). The initial data are $u(x,0)=\frac{0.7660}{1+e^{-x}}$ and $v(x,0)=\frac{0.2977}{1+e^{x}}$.}
	\label{Fig.1}
\end{figure}

 In Example 2,  the coefficient functions are chosen as
\begin{equation}\label{coeff_2}
	\begin{aligned}
		&d_{1}(t)=100, \quad r_{1}(t)=3, \quad \;\;\: a_{1}(t)=3\sin(2t)+14, \quad\;\;\: b_{1}(t)=5\sin(2t)+35;\\
		&d_{2}(t)=120, \quad r_{2}(t)=3.4, \quad a_{2}(t)=1.5\cos(2t)+16, \quad b_{2}(t)=1.5\cos(2t)+6.\\
	\end{aligned}
\end{equation}
If we take $s_{0}=0.81$, then we can verify that (\ref{A2}) (not (\ref{1.12a})), and the conditions (\ref{TH2-cond1}) and (\ref{TH2-cond2}) in Theorem \ref{TH2} hold. The dynamical behavior of the $\pi$-periodic bistable traveling wave is displayed in Fig. 2.
\begin{figure}[H]
	\centering
	\includegraphics[scale=0.60]{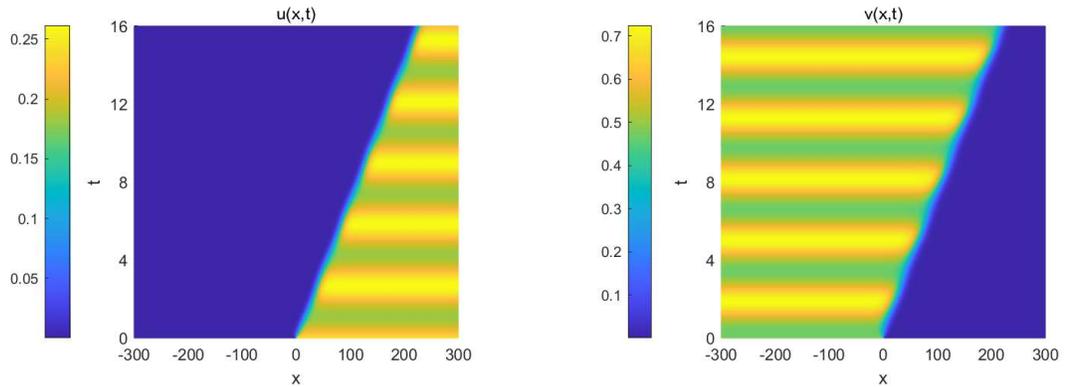}
    \captionsetup{font={small},labelfont=bf,labelformat={default},name={Fig.},labelsep=space}
	\caption{Top views of the $\pi$-periodic bistable wave of (\ref{original_model}). The coefficient functions are shown in (\ref{coeff_2}) and $s_0=0.81$. The initial data are $u(x,0)=\frac{0.2378}{1+e^{-x}}$ and $v(x,0)=\frac{0.4779}{1+e^{x}}$.}
	\label{2}
\end{figure}

\section{Conclusion and discussion}
In this work, we have studied  the Lotka-Volterra type of  competition model with nonlocal dispersal and time periodicity.
By applying the theory of monotone dynamical systems, we prove the existence and monotonicity of the bistable $T$-periodic traveling wave solution. The uniqueness, Lyapunov stability and the value range of the wave speed  have been established mainly by means of the comparison principle  (the upper and lower solution method). Based on these generic results and the characteristics of the bistable waves, we derive explicit conditions for the speed sign, i.e., Theorems \ref{TH1} and \ref{TH2} guarantee the positive and negative  wave speeds, respectively. Moreover, numerical simulations demonstrate our theoretical results even under weak bistable conditions, which reveal the effects of dispersal rate, competition strength, growth rate, seasonality and carrying capacity on the propagation direction of the bistable traveling wave.

It should be pointed out that the monotonicity of the wave speed in terms of the function $b_1(t)$ and $a_2(t)$ can be easily shown by way of comparison principle. However, a complete classification of the speed sign in terms of all parameters is a challenge. As such, we are particularly interested in obtaining analytic and esay-to-apply formulas for determining the speed sign.  Our explicit results are derived by constructing  upper/lower solutions with the asymptotical behavior (\ref{ap1}) which can be seen as case studies,  sheding light on further studies and improvement. We expect that different explicit conditions could be obtained by finding different formulas of upper/lower solutions with the asymptotical behaviors similar to (\ref{ap2}), (\ref{ap3}) and (\ref{ap4}), respectively. The exponential stability of the bistable traveling wave of the system (\ref{original_model}) has been presented in another work. In addition, the condition (\ref{1.12a}) is required only for the existence of traveling waves, while the weaker condition (\ref{A2}) (i.e, the bistable condition) is sufficient for other results. Hence we presume that the existence result developed in  \cite{Fang} (i.e., Lemma \ref{existence lemma}) could be improved.

\section{Data availability}
The simulation code and data are available upon request from the authors. No other data are used. 

\section{Conflict of interest statement}
All authors declare that they have no conflicts of interest.

\vspace{0.5cm}
\textbf{Acknowledgement}.
%The authors would like to thank the referees' valuable comments, which greatly improve the exposition of
%the paper.
The work of Manjun Ma, Wentao Meng and Jiajun Yue was supported  by the National Natural Science Foundation of China (No.~12071434,  No.~11671359).
The work of Chunhua Ou  was  supported by the NSERC discovery grants(RGPIN-2016-04709 and  RGPIN-2022-03842).

\end{document}